\documentclass[12pt,leqno]{amsart}
\usepackage[top=1in, bottom=1in, left=1in, right=1in]{geometry}
\usepackage{amssymb,mathrsfs,amsmath}
\usepackage[alphabetic]{amsrefs}

\usepackage{paralist}
\usepackage{enumitem}

\newcommand{\R}{\mathbb{R}}

\newcommand{\Z}{\mathbb{Z}}

\newcommand{\cA}{\mathcal{A}}
\newcommand{\cB}{\mathcal{B}}
\newcommand{\cD}{\mathcal{D}}
\newcommand{\cE}{\mathcal{E}}
\newcommand{\cH}{\mathcal{H}}
\newcommand{\cJ}{\mathcal{J}}
\newcommand{\cM}{\mathcal{M}}
\newcommand{\cQ}{\mathcal{Q}}
\newcommand{\cU}{\mathcal{U}}

\newcommand{\conj}[1]{\overline{#1}} 

\newcommand{\dint}{\,\mathrm{d}} 

\newcommand{\eps}{\varepsilon} 
\newcommand{\mmid}{\parallel}

\newcommand{\Yil}{Y{\i}ld{\i}r{\i}m}

\newcommand{\cl}[1]{\overline{#1}}

\renewcommand{\leq}{\leqslant}
\renewcommand{\geq}{\geqslant}

\newcommand{\Tn}{T^{\neq}}

\newcommand{\ignore}[1]{}

\DeclareMathOperator*{\sumst}{\sum{}^*}
\DeclareMathOperator*{\dsum}{\sum \sum}
\DeclareMathOperator*{\tsum}{\sum \sum \sum}

\DeclareMathOperator*{\mo}{mod \,}
\DeclareMathOperator{\e}{e}

\theoremstyle{plain}
\newtheorem{thm}{Theorem}
\newtheorem{prop}[thm]{Proposition}
\newtheorem{cor}[thm]{Corollary}
\newtheorem{lem}[thm]{Lemma}

\newtheorem*{primegaps}{Prime Gaps Theorem}
\newtheorem*{mainthm}{Main Theorem}
\newtheorem*{primecousins}{Prime Cousins Theorem}

\theoremstyle{definition}

\newtheorem*{swcond}{S-W Condition}

\theoremstyle{remark}
\newtheorem*{rem}{Remark}
\newtheorem*{rems}{Remarks}

\newtheorem*{examps}{Examples}

\numberwithin{thm}{section}
\numberwithin{equation}{section}
\numberwithin{figure}{section}

\begin{document}

\title{Close encounters among the primes}
\author{J. B. Friedlander and H. Iwaniec}
\date{}
\thanks{Research of JF supported in part by NSERC Grant A5123  
and that of HI supported in part by NSF Grant DMS-1101574.
We thank Leo Goldmakher and Pedro Pontes for their help with the 
physical appearance of the paper.}

\address{J.B.Friedlander: Department of Mathematics, University of Toronto, 
M5S 2E4 Canada}
\address{
H. Iwaniec: Department of Mathematics, Rutgers University, 
New Brunswick NJ 08903, USa}

\begin{titlepage}

\maketitle

\begin{abstract}

This paper was written, apart from one technical correction, 
in July and August of 2013. The, then very recent,
breakthrough of Y. Zhang~\cite{Z} had revived in us an intention to produce 
a second edition of our book ``Opera de Cribro'', one which would include 
an account of Zhang's result, stressing the sieve aspects of the method. A 
complete and connected version of the proof, in our style but not intended 
for journal publication, seemed a natural first step in this project. 

Following the further spectacular advance given 
by J. Maynard (arXiv:1311.4600, Nov 20, 2013), we have had to re-think 
our position. 
Maynard's method, at least in its current form, proceeds from GPY in 
quite a different direction than does Zhang's, and achieves numerically 
superior results. Consequently, although Zhang's contribution to the 
distribution of primes in arithmetic progressions certainly 
retains its importance, the fact remains that much of the material 
in this paper would no longer appear in a new edition of our book. 

Because this paper contains some innovations that we do not wish to become 
lost, we have decided to make the work publicly available in its current form.

\end{abstract} 
\thispagestyle{empty}

\end{titlepage}

\setcounter{tocdepth}{2}

\tableofcontents


\newpage
\section{\bf{Introduction}}

The twin prime conjecture, dating from antiquity, predicts that there are infinitely many pairs of consecutive odd integers, both of which are prime. Letting $p_n$ denote the $n$-th prime number, we can write this as
\begin{equation}
 p_{n+1} - p_n
 = 2
\end{equation}
infinitely often. One knows from the Prime Number Theorem the much weaker statement that
\begin{equation}
 \varliminf_{n \to \infty} \frac{p_{n+1} - p_n}{\log p_n}
 \leq 1,
\label{e12liminfPNT}
\end{equation}
because $\log p_n$ is the average gap. Even replacing the right hand side of~\eqref{e12liminfPNT} by smaller constants took a painfully long time with many contributors, beginning with P. Erd{\H o}s \cite{E} and continuing with \cite{BD}, \cite{H}, \cite{M},~\dots

Two major breakthroughs have occurred during the past few years. In the first of these, D. Goldston, J. Pintz, and C. Y. \Yil \cite{GPY}, building on earlier work \cite{GY}, proved that indeed
\begin{equation}
 \varliminf_{n \to \infty} \frac{p_{n+1} - p_n}{\log p_n} = 0.
\label{e13liminfGY}
\end{equation}
Quite recently, Y. Zhang \cite{Z} has proven that the gaps between consecutive primes are absolutely bounded infinitely often, specifically that
\begin{equation}
 p_{n+1} - p_n
 \leq 70 \, 000 \, 000
\label{e14zhang}
\end{equation}
infinitely often. The proof draws on a wealth of ideas from previous sources while combining significant innovations of its own.

In the authors' book \cite{FI2} we gave an account of the GPY results, following their general outline but introducing some innovations of our own, mainly for the purpose of clarifying the sieve ideas. It is our goal in the current writing to proceed in analogous fashion with the work \cite{Z} of Zhang.

There is a tremendous amount of interest in lowering the numerical bound in~\eqref{e14zhang}, particularly in the pending Polymath8 project. Although we are aware of many possibilities, it is not our intention to pursue them here. Nevertheless, we should specify some bound; we prove the following

\begin{primegaps}
There are infinitely many primes $p_n$ with
\begin{equation}
 p_{n+1} - p_n
 \leq 2 \, 448 \, 798.
\label{e15primeGaps}
\end{equation}
\end{primegaps}

In other words, we have:
\begin{primecousins}
There exists a positive even number $c \leq 2 \, 448 \, 798$ such that $p$ 
and $p+c$ are both primes infinitely often.
\end{primecousins}

Work on the small prime gaps problem has, ever since~\cite{BD}, depended crucially on the Bombieri-Vinogradov theorem for the distribution of prime numbers in residue classes to large moduli. This theorem (which has acted as a substitute for the Generalized Riemann Hypothesis in many applications) gives the bound
\begin{equation}
 \sum_{q \leq Q} \max_{(a,q) = 1} \left|
  \psi(x;q,a) - \frac{x}{\varphi(q)}
  \right|
 \ll x (\log x)^{-A}
\label{e16BomVinDistMod}
\end{equation}
valid for $x \geq 2$, $Q \leq x^{\frac{1}{2}} (\log x)^{-B}$ and any $A \geq 1$, where both $B$ and the implied constant may depend on $A$.

In the fundamental paper \cite{GPY} the authors found that the exponent~$\frac{1}{2}$ was a key threshold. Had it been anything smaller than~$\frac{1}{2}$, their arguments would have failed to obtain~\eqref{e13liminfGY}. On the other hand, could it be replaced by anything larger, say $\frac{1}{2} + \delta$, they could prove the much stronger bounded gaps result
\begin{equation}
 p_{n+1} - p_n
 < c(\delta)
\label{e17liminfGPYcond}
\end{equation}
infinitely often with a positive constant $c(\delta)$ depending on $\delta$.

The surprising, nevermind conditional, result~\eqref{e17liminfGPYcond} of 
Goldston-Pintz-Y{\i}ld{\i}r{\i}m rekindled enhanced interest in 
extending~\eqref{e16BomVinDistMod} type estimates with 
$Q = x^{\frac{1}{2} + \delta}$, going beyond the Riemann Hypothesis 
threshold. There had already been a number of earlier results of 
such quality, by E. Bombieri, E. Fouvry, J. B. Friedlander and 
H. Iwaniec in the series of papers \cite{FoIw}, \cite{F}, 
\cite{BFI1}, \cite{BFI2}, \cite{BFI3}. 
These results did not hold uniformly over all reduced classes $a \pmod{q}$ 
but required one to fix an $a \neq 0$ and then sum over moduli 
$q \sim Q$, $(q,a) = 1$. 
As such, they could not be directly married to the GPY argument.

In the GPY scheme, one considers an admissible set (see~\eqref{e22QDef}-\eqref{e23detHDef})
\[
 \cH
 = \{ h_1,\dots,h_k\}
\]
of positive integers and then applies a sieve to attempt to show that, for many $n$'s, the set $\{n-h_1,\dots,n-h_k\}$ contains more than one prime, hence at least two primes! In applying a sieve to the polynomial $(X - h_1)\cdots(X-h_k)$, one has to consider, in the resulting congruence sums (see p.~125 of \cite{FI2}) more than a single class, yet nowhere near as many as $\varphi(q)$ of them, and these relatively few classes are tied together in arithmetic fashion by the Chinese Remainder Theorem. Even so, the problem of proving a~\eqref{e16BomVinDistMod} type bound for the sum over all moduli $q \sim Q$ in a range $Q = x^{\frac{1}{2} + \delta}$ is still open. Zhang succeeds by considering a sum wherein the moduli run only over smooth numbers, that is, those with no large prime factors, and $a \pmod{q}$ is restricted to the roots of
\[
 \prod_{i \neq j} (X + h_i - h_j)
 \equiv 0 \pmod{q}.
\]
The fact is that, in the GPY sieve, the restriction of the support of 
the sieve weights to smooth numbers does not greatly alter the strength 
of the sieve bounds. This latter phenomenon had been noted and utilized 
by Y. Motohashi and J. Pintz \cite{MP}, and it was also known 
by others. We have articulated the relevant features in various
presentations of sieve capabilities in a broader context. 
Zhang, however, was the one who made the whole 
thing work. The added flexibility in being able to factor the moduli 
in almost arbitrary proportion enabled Zhang to improve the bounds of 
the exponential sums which are required for this approach.

The paper is organized as follows. In Section~\ref{c2GPYRest} we study 
the restricted GPY sieve. We can borrow much of our arguments in this 
part from our book~\cite{FI2}. In Section~\ref{c3proofMainThm} we state 
Proposition~\ref{p34LambdaSumRichEst}, which is the modified extension 
of the bound~\eqref{e16BomVinDistMod}, and we begin its proof. We decompose 
sums over primes by means of combinatorial arrangements as did Zhang (which 
was in turn patterned after that in~\cite{BFI1} where Heath-Brown's formula 
is chosen for the task). This leads us to multiple Dirichlet convolution 
forms which are reduced to two basic types. The first of these, a bilinear 
form, is placed in a general context in Section~\ref{c4bilForms} and 
then estimated in Section~\ref{c5disp} using the dispersion method 
of Linnik \`a la~\cite{FoIw} and~\cite{BFI1}, but with Zhang's modulus 
factorization playing a key role. Then, in Section~\ref{c6CroppDiv} we 
consider smooth croppings, to various sizes, of the divisor function 
$\tau_3(n)$ in residue classes $a \pmod{q}$. The sums of these cropped 
divisors, together with the bilinear forms from Section~\ref{c4bilForms},  
are building blocks for covering the sums over primes. We provide uniform 
estimates 
for single moduli $q$ and every fixed reduced class $a \pmod{q}$ based on 
arguments originating in~\cite{FI1}, but again with a key modification due 
to Zhang which takes advantage of $q$ being a smooth modulus. Actually, one 
could exploit averaging over these special moduli $q$, and this would be 
one of the sources for a lowering of the constant 
in~\eqref{e15primeGaps}, cf. 
E. Fouvry, E. Kowalski, and P. Michel~\cite{FKM2}. 

\ignore{
\medskip
{\bf Comments added, November 30 2013:}

This paper was written in July and August of 2013. The, then very recent,
breakthrough of Y. Zhang had revived in us an intention to produce a second 
edition of our book ``Opera de Cribro'', one which would include an account 
Zhang's result, stressing the sieve aspects of the method. A complete and 
connected version of the proof, in our style but not intended for journal 
publication, seemed a natural first step in this project. 

Following the further spectacular advance given 
by J. Maynard (arXiv:1311.4600, Nov 20, 2013), we have had to re-think 
our position. 
Maynard's method, at least in its current form, proceeds from GPY in 
quite a different direction than does Zhang's, and achieves numerically 
superior results. Consequently, although Zhang's contribution to the 
distribution of primes in arithmetic progressions certainly 
retains its importance, the fact remains that much of the material 
in this paper would no longer appear in a new edition of our book. 

Because this paper contains some innovations that we do not wish to become 
lost, we have decided to make the work publicly available in its original 
form.  
} 

\section{\bf{The GPY Restricted Sieve}}
\label{c2GPYRest}

\subsection{Introduction}

The starting point in Zhang's paper \cite{Z} is the lovely construction of 
the sifted sum by Goldston-Pintz-\Yil, which itself is based on Selberg's 
sieve weights. Zhang applies to their construction the same weights 
$\lambda_d$, but with an extra restriction of the support to ``smooth'' 
numbers, that is; the numbers $d$ which have no large prime divisors, say
\begin{equation}
 d \mid P(z),
\label{e21smoothNum}
\end{equation}
where $P(z)$ denotes the product of all primes $p < z$, and $z$ is relatively 
small. If $z$ is not extremely small, it turns out that such a restriction 
sacrifices very little in the GPY estimates. This phenomenon appears in any 
sieve having large dimension, and it is particularly pronounced in the 
Selberg $\Lambda^2$-sieve, see Section 7.10 of~\cite{FI2}. On the other 
hand, the restriction \eqref{e21smoothNum} introduces a much-desired 
flexibility in the arrangement of bilinear forms out of primes and sieve 
weights. This extra feature permits stronger estimates for the exponential 
sums which control the remainder terms. The gain amply compensates for the 
loss caused by~\eqref{e21smoothNum}.

Much of the material required exists already in the book~\cite{FI2} in 
thinly disguised form; therefore, we shall make frequent reference to this 
material. We begin by recalling the GPY construction (in our format).

Let $\cH$ be a finite set of $k \geq 2$ positive integers,
\begin{equation}
 Q(X)
 = \prod_{h \in \cH} (X - h)
\label{e22QDef}
\end{equation}
and $\omega(d)$ the number of roots of $Q(x) \equiv 0 \pmod{d}$. We say that 
the set $\cH$ is \emph{admissible} if $\omega(p) < p$ for every prime $p$. 
In particular, it shows that $\cH$ does not contain two consecutive integers. 
Let $\Delta$ be a positive squarefree integer which is divisible by every 
prime factor of
\begin{equation}
 \det \cH
 = \prod_{h \neq h'} (h - h')
 = \prod_{h \in \cH} |Q'(h)|.
\label{e23detHDef}
\end{equation}
Every prime $p \leq k$ divides $\Delta$. If $p \nmid \Delta$, then 
$\omega(p) = k$. Associated with the set $\cH$, we introduce two positive 
constants
\begin{equation}
 \gamma(\cH)
 = \prod_{p \mid \Delta} \left( 1 - \frac{\omega(p)}{p} \right),
  \quad
 H(\cH)
 = \prod_p \left(1 - \frac{\omega(p)}{p} \right) 
\left(1 - \frac{1}{p} \right)^{-k}.
\end{equation}

We consider the sequence $\cA = (a_n)$ with
\begin{equation}
 a_n
 = \sum_{h \in \cH} \Lambda(n - h)
  - \log n
\label{e25defan}
\end{equation}
for $n > \max \cH$. Our goal is to prove that $a_n$ is positive for infinitely 
many $n$. For every such $n$ it follows that at least two of the shifted 
integers $n - h$ with $h \in \cH$ are prime powers. Actually, the higher 
powers will make a negligible contribution, so we obtain gaps between primes 
bounded by the diameter of~$\cH$. We are not able to achieve this goal by 
summing the elements themselves, but can do so by weighting the sequence 
$\cA = (a_n)$ in the style which is familiar from Selberg's lower-bound 
sieve (see Section~7.6 of~\cite{FI2}). Specifically, we consider the 
following sum:
\begin{equation}
 W(D,N)
 = \sum_{(Q(n),\Delta) = 1} F(n / N) a_n
  \bigg( \sum_{d \mid Q(n)} \lambda_d \bigg),
\label{e26WeLB}
\end{equation}
where $(\lambda_d)$ is an upper-bound sieve of level $D$. We shall specify 
the sieve later and it will be a $\Lambda^2$-sieve. Here $F(t)$ is any fixed 
smooth function compactly supported in $\R^+$, $F(t) \geq 0$. Therefore the 
summation over $n$ in~\eqref{e26WeLB} is restricted to a segment $n \asymp N$. 
Of course, this smoothing device is introduced only for technical 
simplification; it has no bearing on the GPY construction.

We shall prove that, for a judiciously chosen sieve $(\lambda_d)$,
\begin{equation}
 W(D,N)
 \asymp N (\log N)^{1 - k}
\label{e27WAsymp}
\end{equation}
provided $k = |\cH|$ is sufficiently large; $k = 175561$ is fine. Hence, 
we shall derive (see Section~\ref{s34ProofsThms}) the following results

\begin{mainthm}
For all $N$ sufficiently large in terms of the set $\cH$ we have
\begin{equation}
 \sum_{\substack{N < n < 2N \\ (Q(n),P(w)) = 1}} \mathop{\sum \sum}_{h \neq h'}
  \Lambda(n - h) \Lambda(n - h')
 \asymp N (\log N)^{2-k} , 
\label{e28MainThm}
\end{equation}
where $w = N^{1/b}$ with $b$ a sufficiently large constant.
\end{mainthm}

The extra restriction $(Q(n),P(w)) = 1$ in~\eqref{e28MainThm} means that 
every number $n - h$ with $h \in \cH$ has no prime divisors smaller than 
$N^{1/b}$, hence it has at most $b$ prime factors. Ignoring this restriction 
we deduce from~\eqref{e28MainThm} that there exists 
$c \in \cH - \cH$, $c \neq 0$, such that $\pi_c(N)$, the number 
of primes $p$ in $N \le p < 2N$ with $p+c$ being a prime, satisfies 
the lower bound
\begin{equation}
 \pi_c(N)
 \gg N(\log N)^{-k}
\end{equation}
for a positive proportion of $N$'s.
A more precise expectation, the Hardy-Littlewood conjecture, predicts that 
for any even number $c \neq 0$ we have 
$\pi_c(N) \sim \mathfrak S N(\log N)^{-2}$ as $N\rightarrow\infty$,
where $\mathfrak S$ is a positive constant depending on $c$.


\subsection{Shuffling the Sieve Terms}

Given any sieve-weighted sum, its terms need to be rearranged so that the 
evaluation job boils down to that of the congruence sums for the sifting 
sequence. In this subsection we do so with~\eqref{e26WeLB} for the 
sequence \eqref{e25defan}. We also execute some easy parts of the relevant 
summations to clean and isolate those things which are more intricate. 
Throughout, we assume that $\lambda_d$ are supported on squarefree numbers 
$d < D$, $(d,\Delta) = 1$ and
\begin{equation}
 |\lambda_d|
 \leq \tau_3(d).
\label{e210LambdaBd}
\end{equation}
First, inserting~\eqref{e25defan} into~\eqref{e26WeLB} we get
\begin{equation}
 W(D,N)
 = \sum_{h \in \cH} V_h(D,N)
  - U(D,N) \log N
\label{e211WbyVhU}
\end{equation}
where we set
\begin{equation}
 V_h(D,N)
 = \sum_{(Q(n),\Delta)=1} \Lambda (n - h) F\left(\frac{n}{N}\right)
  \bigg( \sum_{d \mid Q(n)} \lambda_d \bigg) , 
\end{equation}
\begin{equation}
 U(D,N)
 = \sum_{(Q(n),\Delta)=1} F\left(\frac{n}{N}\right) \frac{\log n}{\log N}
  \bigg( \sum_{d \mid Q(n)} \lambda_d \bigg).
\end{equation}
The latter sum can be evaluated quickly by splitting into residue classes 
$\alpha \pmod{\Delta}$ and $\beta \pmod{d}$. We have
\begin{align*}
 \sum_{\substack{n \equiv \alpha \, (\Delta) \\ n \equiv \beta \, (d)}}
  F \left( \frac{n}{N} \right) \frac{\log n}{\log N}
 &= \frac{1}{\Delta d}
   \int F\left(\frac{t}{N}\right) \frac{\log t}{\log N} \dint t
   + O(1)
\\
 &= \frac{N}{\Delta d}
   \int F(t) \left(1 + \frac{\log t}{\log N} \right) \dint t
   + O(1) 
\end{align*}
 (throughout, we shall ignore the dependence on $F$ in 
any implied constants). 
Moreover, the number of classes $\alpha \pmod{\Delta}$ with 
$(Q(\alpha),\Delta) = 1$ is equal to $\gamma(\cH)\Delta$ and the number 
of classes $\beta \pmod{d}$ with $d \mid Q(\beta)$ is equal to $\tau_k(d)$. 
Hence
\begin{equation}
 U(D,N)
 = \gamma(\cH) G(D) \big\{
   \hat{F}(0) N + \tilde{F}(0) N / \log N
   \big\}
  + O \big( D ( \log D)^{3k} \big)
\end{equation}
where $\hat{F}(0)$, $\tilde{F}(0)$ are the integrals of $F(t)$, $F(t) \log t$, respectively, and
\begin{equation}
 G(D)
 = \sum_{(d,\Delta)=1} \lambda_d \tau_k(d) / d.
\label{e215GDef}
\end{equation}

Next, we are going to evaluate $V_h(D,N)$ for every $h$ separately. 
Given $h \in \cH$, let $Z_h(X)$ denote the polynomial
\begin{equation}
 Z_h(X)
 = X^{-1} Q(X+h)
 = \prod_{h' \neq h} (X + h - h').
\end{equation}
By shifting $n$ to $n+h$ and using $F((n+h)/N) = F(n/N) + O(h/N)$, we get
\[
 V_h(D,N)
 = \sum_{(nZ_h(n),\Delta)=1} \Lambda(n) F(n/N)
  \sum_{d \mid Z_h(n)} \lambda_d
  + O \big( D (\log D)^{3k} \big).
\]
Note that the contribution of $n$'s which are not coprime with $d$ is negligible, because $n$ is a prime power and $d < D$. Splitting the summation over $n$ into residue classes $\alpha \pmod{\Delta}$, $\beta \pmod{d}$ such that
\begin{equation}
 \big( \alpha Z(\alpha), \Delta \big) = 1,
  \quad
 (\beta,d) = 1,
  \quad
 d \mid Z_h(\beta) , 
\label{e217splitAlphaBeta}
\end{equation}
we obtain
\[
 \sum_d \lambda_d
 \sum_{\alpha \, (\Delta)}
 \sum_{\beta \, (d)}
 \sum_{\substack{n \equiv \alpha \, (\Delta) \\ n \equiv \beta \, (d)}} \Lambda(n) F(n/N).
\]
If the inner sum is replaced by its expected value
\begin{equation}
 \frac{1}{\varphi(q)}
 \sum_{(n,q) = 1} \Lambda(n) F(n/N),
  \qquad
 q = \Delta d,
\label{e218innSumExp}
\end{equation}
we get a free summation over all $\alpha,\beta$ satisfying~\eqref{e217splitAlphaBeta}. The number of $\alpha$'s is the same as before, namely $\gamma(\cH)\Delta$, but the number of $\beta$'s drops to $\tau_{k-1}(d)$. Moreover, \eqref{e218innSumExp} is well-approximated by $\hat{F}(0)N / \varphi(q)$ up to a very small error term $(\log q)/\varphi(q)$. Hence we can write
\begin{equation}
 V_h(D,N)
 = \gamma(\cH) G'(D) \hat{F}(0) N
  + R_h(D,N)
  + O \big( D(\log D)^{3k} \big),
\label{e219VhAppr}
\end{equation}
where
\begin{equation}
 G'(D)
 = \frac{\Delta}{\varphi(\Delta)}
  \sum_{(d,\Delta)=1} \lambda_d \tau_{k-1}(d) / \varphi(d)
\label{e220G'Exp}
\end{equation}
and $R_h(D,N)$ is the remainder
\begin{equation}
  \sumst_{\substack{\alpha\,(\Delta) \\ (Z_h(\alpha),\Delta)=1}}
  \sum_{d < D} \lambda_d
  \sumst_{\substack{\beta \, (d)\\d \mid Z_h(\beta)}} \Bigg(
   \sum_{\substack{n \equiv \alpha \, (\Delta)\\n \equiv \beta \, (d)}}
   \Lambda(n) F\left(\frac{n}{N}\right)
   - \frac{1}{\varphi(\Delta d)}
   \sum_{(n,\Delta d) = 1} \Lambda(n) F\left(\frac{n}{N}\right)
   \Bigg).
\label{eR_h}
\end{equation}

Note that the main term in~\eqref{e219VhAppr} is the same one for every 
$h \in \cH$, and we have $k = |\cH|$ of these terms. Adding $V_h(D,N)$ and 
subtracting $U(D,N)$ as in~\eqref{e211WbyVhU}, we arrive at the following 
formula for the GPY sifted sum: 
\begin{multline}
 W(D,N)
 = \gamma(\cH) N \hat{F}(0) \Big\{
   k G'(D)
   - G(D) \big( \log N + O(1) \big)
   \Big\}
\\
  + \sum_{h \in \cH} R_h(D,N)
  + O\big( D(\log D)^{3k} \big).
\end{multline}

The remainders $R_h(D,N)$, which are given by ~\eqref{eR_h} must 
satisfy, for every
$h \in \cH$, 
\begin{equation}
 R_h(D,N)
 \ll N (\log N)^{-A}
\label{e223RemBd}
\end{equation}
with some $A > k-1$ (for practical purposes, for every $A$), so they do not 
strike the main term. It is critical to have~\eqref{e223RemBd} for 
$D = N^\theta$ with $\theta > \frac{1}{2}$, and we shall achieve this 
level of moduli, but only for a special choice of the sieve weights 
$\lambda_d$.

At the same time, we also require the upper-bound 
sieve $(\lambda_d)$ of level $D$ to be such that
\begin{equation}
 k G'(D) - G(D) \log N
 > \eta (\log N)^{1 - k},
\label{e224SieveReq}
\end{equation}
with a small positive constant $\eta$, so that the lower bound 
in~\eqref{e27WAsymp} will hold (the corresponding upper bound follows 
easily by any reasonable sieve of dimension $k -1$).


\subsection{Choosing the Sieve Weights}

So far our transformations are valid for any upper-bound sieve $(\lambda_d)$ 
of level $D$, actually for any sequence $(\lambda_d)$ of real numbers 
satisfying~\eqref{e210LambdaBd}, but we are interested in special 
$(\lambda_d)$ which yield~\eqref{e224SieveReq}. Here we are faced with 
sieves of dimension $k-1$ in the case of $G'(D)$ and of dimension $k$ in 
the case of $G(D)$. It turns out that $k$ must be quite large to give the 
positive lower bound~\eqref{e224SieveReq}, and the $\Lambda^2$-sieve of 
Selberg works best in large dimensions. The $\Lambda^2$-sieve assumes 
that $\lambda_d$ are given in the form
\begin{equation}
 \lambda_d
 = \sum_{[d_1,d_2]=d} \rho_{d_1} \rho_{d_2}
\label{e225lambdaSelberg}
\end{equation}
where $(\rho_d)$ is any sequence of real numbers with
\begin{equation}
 \rho_1 = 1,
  \qquad
 \rho_d = 0
  \quad \text{if } d \geq \sqrt{D}.
\end{equation}
For a complete account of the $\Lambda^2$-sieve, see Chapter 7 of~\cite{FI2}. 
Now, $kG'(D) - G(D)$ becomes a quadratic form in the sieve constituents 
$\rho_d$, $1 \leq d < \sqrt{D}$. It would be a complicated process to 
determine the maximum of this quadratic form, however, it suffices to 
have results which are manageable in applications. To this end we follow 
the strategy of minimizing $G(D)$ alone, hoping that the resulting value 
of $G'(D)$ will be pretty good as well. Let us say a few words about how 
we proceeded in~\cite{FI2} with the arithmetical transformations. First, 
by the linear substitution (7.180), the quadratic form $G(D)$ is 
diagonalized to (7.181) while $G'(D)$ is expressed in the new variables 
in (7.186). At this point, we make the special choice (7.187) of the new 
variables $y_d$: 
\[
 y_d
 = \frac{1}{Y(D)} \left( \log \frac{\sqrt{D}}{d} \right)^l,
  \qquad
 1 \leq d < \sqrt{D},
\]
where $l$ is any positive integer at our disposal and $Y(D)$ is the normalizing factor. This choice yields (7.188) and (7.189) with $Y(D)$ given by (7.190). We write
\[
 Y(D)
 = \sum_{m < \sqrt{D}} h(m) \left( \log \frac{\sqrt{D}}{m} \right)^l
\]
where $h(d)$ is the multiplicative function supported on squarefree numbers $d$ coprime with $\Delta$, $h(p) = k / (p-k)$ if $p \nmid \Delta$. Hence, the original sieve constituents become
\[
 \rho_d
 = \mu(d) \frac{d}{\tau_k(d)} \sum_{\substack{b < \sqrt{D} \\ b \equiv 0 \, (d)}} h(b) y_b
 = \frac{\mu(d) d}{Y(D) \tau_k(d)} \sum_{\substack{d < \sqrt{D} \\ b \equiv 0 \, (d)}}
  h(b) \left( \log \frac{\sqrt{D}}{b} \right)^l,  
\]
see (7.182) of~\cite{FI2}. By the arguments of van Lint-Richert (see the bottom lines on p.~92 of~\cite{FI2}) we get
\begin{align*}
 Y(D)
 &= \sum_{a \mid d} \sum_{\substack{m < \sqrt{D} \\ (m,d) = a}} h(m) \left( \log \frac{\sqrt{D}}{m} \right)^l
  = \sum_{a \mid d} h(a) \sum_{\substack{m < \sqrt{D}/a \\ (m,d) = 1}} h(m) \left( \log \frac{\sqrt{D}}{am} \right)^l
\\
 &\geq \left(\sum_{a \mid d} h(a)\right)
   \sum_{\substack{m < \sqrt{D}/d \\ (m,d)=1}} h(m) \left(\log \frac{\sqrt{D}}{dm} \right)^l
  = \mu(d) \rho_d Y(D).
\end{align*}
Hence $|\rho_d|\leq 1$, which, together with~\eqref{e225lambdaSelberg} 
implies~\eqref{e210LambdaBd}.

\begin{rems}
The above choice of $\rho_d$ is not optimal for $G(D)$, however it is used in the original work of GPY because it produces a very good value of the whole quadratic form $k G'(D) - G(D)$. After the GPY work, B. Conrey considered $\mu(d) \rho_d$ as a continuous function of $d$ and applied variational calculus to find the best choice, which turned out to be a combination of two Bessel functions $J_{k-2}$ (in unpublished notes of 2005).
\end{rems}

We have evaluated the sums $G(D),G'(D)$ asymptotically in (7.193), (7.194) 
respectively, again from~\cite{FI2}. Here, we quote these results: 
\begin{align}
 G(D)
 &= \gamma(k,l) \frac{H(\cH)}{\gamma(\cH)} \big(\log \sqrt{D}\big)^{-k} \left\{
   1 + O \left( \frac{1}{\log D} \right)
   \right\},
\\
 G'(D)
 &= \frac{2l+1}{(l+1)(k+2l+1)} G(D) \big\{
  \log D + O(1)
  \big\},
\end{align}
where $\gamma(k,l) = k! \binom{2l}{l} \binom{l+k}{l} / \binom{2l+k}{l}$. Hence, (2.22) becomes
\begin{equation}
 W(D,N)
 = \gamma(k,l) H(\cH) \hat{F}(0) \big\{ w(D,N) + O(1) \big\} N \big(\log \sqrt{D}\big)^{-k}
  + \sum_{h \in \cH} R_h(D,N)
\label{e229Wexpw}
\end{equation}
where
\begin{equation}
 w(D,N)
 = \mu(k,l) \log D - \log N
\end{equation}
with $\mu(k,l) = k (2l+1) / (l+1)(k + 2l + 1)$, which we write as
\begin{equation}
 \mu(k,l)
 = 2 \bigg/ \left( 1 + \frac{1}{2l+1} \right) \left( 1 + \frac{2l+1}{k} \right).
\end{equation}

It is clear that the factor $\mu(k,l)$ is always smaller than $2$, so we need the sieve level $D = N^\theta$ with $\theta > \frac{1}{2}$ to be able to conclude that $w(D,N) \gg \log N$, and hence~\eqref{e27WAsymp}. The Bombieri-Vinogradov theorem can handle our remainders $R_h(D,N)$ only to level $D < N^{\frac{1}{2}}$, so it is not sufficient, yet it barely misses. Indeed, if $l$ is large and $k$ is considerably larger, then $\mu(k,l)$ is close to $2$. A nice choice of $l$ in terms of $k$ is the integer with $\sqrt{k} \leq 2l+1 < \sqrt{k} + 2$, giving
\[
 \mu(k,l)
 \geq 2 \bigg/ \left( 1 + \frac{2}{\sqrt{k}} + \frac{2}{k} \right).
\]

As alluded in the introduction, we shall be able to achieve the sieve level $D = N^\theta$ with an absolute exponent $\theta > \frac{1}{2}$ for the sieve weights $\lambda_d$ supported on numbers having no large prime divisors, say
\begin{equation}
 d \mid P(z),
  \quad
 z = D^{1 / 2s},
  \quad
 s \geq 1.
\label{e232NoLargePrimeDiv}
\end{equation}
Let $G_s(D),G_s'(D)$ denote the corresponding sums~\eqref{e215GDef}, \eqref{e220G'Exp} with the extra restriction~\eqref{e232NoLargePrimeDiv}. There is no extra trouble to evaluate these sums. The extra restriction~\eqref{e232NoLargePrimeDiv} can be automatically secured by assuming (temporarily) in the arithmetical transformations (7.180)-(7.194) of~\cite{FI2}, that $\Delta$ is divisible by every prime in the segment $z \leq p < \sqrt{D}$. Of course, such a $\Delta$ would be huge (temporarily), but it does not matter as long as we are handling algebraic identities without estimating anything. Therefore (7.188) of \cite{FI2} becomes
\begin{equation}
 Y_s^2(D) G_s(D)
 = \sum_{\substack{d < \sqrt{D}\\ \Delta d \mid P(z)}} \frac{\tau_k(d)}{f(d)} \left( \log \frac{\sqrt{D}}{d} \right)^{2l}
\end{equation}
and (7.189) of~\cite{FI2} becomes
\begin{equation}
 Y_s^2(D) G_s'(D)
 =
 \frac{\Delta}{\varphi(\Delta)}
  \sum_{\substack{d < \sqrt{D}\\ \Delta d \mid P(z)}}
   \bigg( \frac{d}{\varphi(d)} \bigg)^2 
   \frac{\tau_{k-1}(d)}{f(d)}
   \Bigg(
    \sum_{\substack{n < \sqrt{D}/d \\ \Delta d n \mid P(z)}}
    \frac{1}{\varphi(n)}
    \bigg( \log \frac{\sqrt{D}}{dn} \bigg)^l
   \Bigg)^2
\label{e234Ys2Gs'}
\end{equation}
with $Y_s(D)$, the normalizing factor, given by
\begin{equation}
 Y_s(D)
 = \sum_{\substack{d < \sqrt{D} \\ \Delta d \mid P(z)}}
  \frac{\tau_k(d)}{f(d)} \left( \log \frac{\sqrt{D}}{d} \right)^l.
\end{equation}
Here $f(d)$ is the multiplicative function with $f(p) = p-k \geq 1$, and $l \geq 0$ is an integer at our disposal. For $z = \sqrt{D}$, that is, for $s = 1$, we have the formula (7.192) of~\cite{FI2};
\begin{equation}
 Y_1(D)
 = Y(D)
 = \frac{l!}{(l+k)!} \frac{\gamma(\cH)}{H(\cH)} \big(\log \sqrt{D}\big)^{k+l} \left\{
  1 + O \left( \frac{1}{\log D} \right)
  \right\}.
\end{equation}

In the next subsection we shall estimate $G_s'(D)$ from below for all 
$s \geq 1$. Of course, $G_s'(D) \leq G_1'(D) = G'(D)$, but it turns out 
that $G_s'(D)$
 loses very little if $s$ is relatively small; precisely we have 
\begin{equation}
 G_s'(D)
 \geq \left(1 - 2s \left(1 - \frac{1}{s} \right)^k \right) G'(D),
  \quad \text{if } s \geq 2.
\label{e237Gs'LB}
\end{equation}
Note that $(1-1/s)^k < e^{-k/s}$. For $G_s(D)$ and $Y_s(D)$ we can use the 
obvious upper bounds $G_s(D) \leq G_1(D) = G(D)$ and 
$Y_s(D) \leq Y_1(D) = Y(D)$. Hence, for the sum~\eqref{e26WeLB} with 
$\lambda_d$ restricted by~\eqref{e232NoLargePrimeDiv}, we 
obtain~\eqref{e229Wexpw} with
\begin{equation}
 w(D,N)
 \geq C_k(\theta,s) \log N
\label{e238wLB}
\end{equation}
provided
\begin{equation}
 C_k(\theta,s)
 = 2 \theta (1 - 2 s e^{-k/s}) \left(
  1 + \frac{2}{\sqrt{k}} + \frac{2}{k}
  \right)^{-1}
  - 1
 > 0.
\label{e29CkPos}
\end{equation}

If $z = D^{1/2s}$ with $s$ sufficiently large, that is, if the support of the sieve $(\lambda_d)$ is smooth enough, we shall achieve a level $D = N^{\theta}$ with $\theta = \theta(s)$ exceeding $\frac{1}{2}$ by an absolute constant. Then, taking $k$ large, it is evident without resorting to numerical computation, that $C_k(\theta,s)$ is positive. We shall be allowed to take (see~\eqref{e335DeltaCond}) $s = 4494$ and $\theta = 105/209$, numbers which are sufficient to show the positivity of $C_k(\theta,s)$ for
\[
 k
 = (419)^2
 = 175561.
\]
Note that for this choice the loss of $2 s (1 - 1/s)^k$ in the lower bound~\eqref{e237Gs'LB}, which is caused by the sieve-weight restriction~\eqref{e232NoLargePrimeDiv}, is numerically minuscule, $2 s e^{-k/s} < 10^{-13}$.

It remains to prove the lower bound~\eqref{e237Gs'LB}, which we 
do in the next subsection.


\subsection{Estimation of $G_s'(D)$}
\label{Gs'(D)}

Recall that $G_s'(D)$ is given by~\eqref{e234Ys2Gs'} with $z = D^{1/2s}$ and 
that $G'(D) = G_1'(D)$. We begin the proof of~\eqref{e237Gs'LB} by 
generalizing Corollary~A.6 of~\cite{FI2} to accommodate the restriction 
$m \mid P(z)$.

Let $\kappa \geq 1$ and $l \geq 0$ be integers and $g(m)$ be a multiplicative function supported on squarefree integers such that
\begin{equation}
 0
 \leq g(p)
 = \frac{\kappa}{p} + O \left(\frac{1}{p^2}\right).
\end{equation}
Denote, for $x>1$, $z\geq 2$, 
\begin{equation}
 S_l(x,z)
 = \sum_{\substack{m \leq x \\ \Delta m \mid P(z)}} g(m) \left( \log \frac{x}{m} \right)^l.
\label{e241SlDef}
\end{equation}
In Corollary~A.6 of~\cite{FI2} we proved that for $x = z \geq 2$ (in which case the summation is restricted only by $(m,\Delta) = 1$)
\begin{equation}	
 S_l(x,x)
 = C(\log x)^{\kappa + l} \big\{ 1 + O ( 1 / \log x ) \big\}
\label{e242Slxx}
\end{equation}
where $C$ is a positive constant which depends on $\Delta,l$ 
and the function $g$. Hence, we derive: 

\begin{lem}
For $x > 1$ and $z \geq 2$ we have
\begin{equation}
 S_l(x,z)
 = h(s) S_l(z,z) \big\{ 1 + O ( 1 / \log z ) \big\},
\label{e243Slxzzz}
\end{equation}
where $s = \log x / \log z$ and $h(s)$ is the continuous solution 
to the differential-dif\-fer\-ence equation
\begin{align}
 h(s)
 &= s^{\kappa + l} ,
  & & \text{if } 0 < s \leq 1 , 
\label{e244DifDifCond}
\\
 sh'(s)
 & = (\kappa + l) h(s) - \kappa h(s-1) ,
  & & \text{if } s > 1.
\label{e245DifDifEq}
\end{align}
\end{lem}

\begin{proof}
For $l = 0$ we get~\eqref{e243Slxzzz} by Theorem~A.8 of~\cite{FI2}, which we write in the form
\[
 S_0(x,z)
 = h_0(s) S_0(z,z) \big\{
  1 + O(1/\log z)
  \big\},
\]
where $h_0(s)$ is the solution of~\eqref{e244DifDifCond}-\eqref{e245DifDifEq} with $l = 0$. We derive the general case $l \geq 1$ by partial summation as follows;
\begin{align*}
 S_l(x,z)
 &= \int_1^x \left(\log \frac{x}{y}\right)^l \dint S_0(y,z)
  = - \int_1^x S_0(y,z) \dint \left( \log \frac{x}{y} \right)^l
\\
 &= - \int_1^x \left\{
   h_0 \left( \frac{\log y}{\log z} \right)
   + O\left(\frac{1}{\log z}\right)
   \right\}
   \dint \left( \log \frac{x}{y}\right)^l S_0(z).
\end{align*}
Changing the variable $y = z^t = x^{t/s}$ we get
\begin{equation}
 S_l(x,z)
 = I(s) S_0(z,z) (\log z)^l \big\{ 1 + O (1/\log z) \big\},
\label{e246SlByI}
\end{equation}
where
\begin{equation}
 I(s)
 = - \int_0^s h_0(t) \dint (s-t)^l.
\label{e247Idef}
\end{equation}
For $0 < s \leq 1$ this gives
\begin{equation}
 I(s)
 = l \int_0^s t^\kappa (s-t)^{l-1} \dint t
 = \frac{\kappa! l!}{(\kappa + l)!} s^{\kappa+l}
 = I(1) h(s).
\end{equation}
For $s > 1$ we write the integral~\eqref{e247Idef} in the form
\[
 s I(s)
 = l \int_0^s h_0(t) (s-t)^l \dint t
  + l \int_0^s h_0(t) t (s-t)^{l-1} \dint t.
\]
Differentiating, we derive (assume $l > 1$) 
\[
 \big(sI(s)\big)'
 = l I(s)
  - l \int_0^s h_0(t) t \dint (s-t)^{l-1}
 = l I(s)
  + l \int_0^s \big( h_0(t) t \big)' (s-t)^{l-1} \dint t.
\]
Hence, 
\begin{align*}
 \big(sI(s)\big)'
 &= l I(s)
   + l \int_1^s \big( (\kappa +1) h_0(t) - \kappa h_0(t-1) \big) (s-t)^{l-1} \dint t
+ l \int_0^1 \big(h_0(t) t \big)' (s-t)^{l-1} \dint t
\\
 &= (\kappa + l + 1) I(s) - \kappa I(s-1).
\end{align*}
Thus, $I(s)$ satisfies the differential-difference equation~\eqref{e245DifDifEq}. (For $l = 1$ the arguments are similar.) Then, \eqref{e246SlByI} with $x = z \geq 2$ yields
\[
 S_l(z,z)
 = I(1) S_0(z,z)(\log z)^l \big\{ 1 + O (1/\log z) \big\}.
\]
Combining this with~\eqref{e246SlByI} we obtain~\eqref{e243Slxzzz}.
\end{proof}

Note that~\eqref{e243Slxzzz} can be written in the following form (use~\eqref{e242Slxx})
\begin{equation}
 S_l(x,z)
 = H(s) S_l(x,x) \big\{ 1 + O(1/\log x) \big\},
  \quad \text{if } x \geq 2,
\label{e249SlxzBySlxx}
\end{equation}
where $H(s) = s^{-\kappa-l} h(s)$, so $H(s)$ is the continuous solution of
\begin{align}
 H(s)
 &= 1 , 
 & & \text{if } 0 < s \leq 1 , 
\label{e250DifDefHBdry}
\\
 s^{\kappa + l + 1} H'(s)
 &= - \kappa (s-1)^{\kappa + l} H(s-1),
 & & \text{if } s > 1.
\label{e251DifDifH}
\end{align}
Obviously, $H(s)$ is positive and decreasing, so $H(s) \leq 1$.

\begin{lem}
For $s > 1$ we have
\begin{equation}
 H(s)
 > 1 - \frac{\kappa s}{\kappa + l + 1} \left(1 - \frac{1}{s}\right)^{\kappa + l + 1}.
\label{e252HLowBd}
\end{equation}
\end{lem}

\begin{proof}
By~\eqref{e251DifDifH} we get
\begin{align*}
 1 - H(s)
 &= \kappa \int_1^s \left(1 - \frac{1}{t}\right)^{\kappa + l} t^{-1} H(t-1) \dint t
  \leq \kappa \int_1^s \left(1 - \frac{1}{t}\right)^{\kappa + l} t^{-1} \dint t
\\
 &\leq \kappa s \int_1^s \left(1 - \frac{1}{t}\right)^{\kappa + l} t^{-2} \dint t
  = \frac{\kappa s}{\kappa + l + 1} \left( 1 - \frac{1}{s} \right)^{\kappa + l + 1}.
\qedhere
\end{align*}
\end{proof}

\begin{rems}
One can show that $H(s)$ satisfies the integral equation
\begin{equation}
 s^{\kappa + l + 1} q(s) H(s)
 = c + \kappa \int_{s-1}^s x^{\kappa + l} q(x+1) H(x) \dint x
\label{e253HintEq}
\end{equation}
for $s \geq 1$, where $c = e^{\kappa \gamma}(\kappa + l)! / l!$, and
\begin{equation}
 q(s)
 = \frac{1}{l!} \int_0^\infty \exp \left(
  -s z + \kappa \int_0^z (1 - e^{-u}) u^{-1} \dint u
  \right)
  z^l \dint z.
\end{equation}
This is the solution of the equation
\begin{equation}
 \big( s^{\kappa + l + 1} q(s) \big)'
 = \kappa s^{\kappa + l} q(s+1),
  \qquad
 s > 0,
\end{equation}
which is adjoint to the equation~\eqref{e251DifDifH},
 see Appendix B of~\cite{FI2}.  
We have, by (B.13) and (B.12) of~\cite{FI2})
\begin{equation}
\begin{aligned}
 q(s)
 &\sim c s^{-\kappa - l - 1},
 & & \text{as } s \to 0,
\\
 q(s)
 &\sim s^{-l -1},
 & & \text{as } s \to \infty.
\end{aligned}
\end{equation}
Hence~\eqref{e253HintEq} shows that
\begin{equation}
 H(s)
 \sim c s^{-\kappa},
  \qquad
 \text{as } s \to \infty.
\end{equation}
Our neat lower bound~\eqref{e252HLowBd} can be improved significantly if $s > \kappa + l$ by using~\eqref{e253HintEq}, but we do not need anything in this range.
\end{rems}

Now we are ready to prove~\eqref{e237Gs'LB} for any $s \geq 2$. By~\eqref{e241SlDef}, \eqref{e242Slxx}, and~\eqref{e249SlxzBySlxx}, the innermost sum in~\eqref{e234Ys2Gs'} is equal to
\[
 C H\left( \frac{\log \sqrt{D}/d}{\log z} \right)
 \left( \log \frac{\sqrt{D}}{d}\right)^l \left(
  \log \frac{\sqrt{D}}{d} + O(1)
  \right)
\]
where $H(s)$ is the function defined in \eqref{e250DifDefHBdry} and \eqref{e251DifDifH} for $\kappa = 1$ (in the range $\sqrt{D}/2 < d < \sqrt{D}$, the claim is trivial). Hence $ Y_s^2(D) G_s'(D)$ in~\eqref{e234Ys2Gs'} is equal to 
\[
 \frac{C^2 \Delta}{\varphi(\Delta)}
 \sum_{\substack{d < \sqrt{D} \\ \Delta d \mid P(z)}}
 \bigg( \frac{d}{\varphi(d)}\bigg)^2
 \frac{\tau_{k-1}(d)}{f(d)}
 H^2 \bigg( \frac{\log \sqrt{D}/d}{\log z}\bigg)
\bigg( \log \frac{\sqrt{D}}{d} \bigg)^{2l}
 \bigg( \log \frac{\sqrt{D}}{d} + O(1) \bigg)^2.
\]
All we need to know about $H(s)$ is~\eqref{e250DifDefHBdry} and the lower bound
\[
 H^2(s)
 > 1 - \frac{2s}{l + 2} \left(1 - \frac{1}{s}\right)^{l+2}
  \qquad \text{if } s > 1,
\]
which follows by~\eqref{e252HLowBd}. Hence, the above 
formula yields the following lower bound:
\begin{equation}
 Y_s^2(D) G_s'(D)
 \geq \frac{C^2 \Delta}{\varphi(\Delta)}
  \sum_{\substack{d < \sqrt{D} \\ \Delta d \mid P(z)}}
  \left( \frac{d}{\varphi(d)} \right)^2 \frac{\tau_{k-1}(d)}{f(d)} \{ \cdots \}
   + O \big( (\log D)^{k + 2l} \big)
\label{eYsGd'LB}
\end{equation}
where the content of $\{ \cdots \}$ is equal to
\[
 \bigg( \log \frac{\sqrt{D}}{d}\bigg)^{2l +2}
 - \,\,\,\, \frac{2}{l+2} (\log z)^{-1} \bigg( \log \frac{\sqrt{D}}{d} \bigg)^{l+1} 
\bigg( \log \frac{\sqrt{D}}{zd}\bigg)^{l+2}
\]
with the negative term being present only for $d < \sqrt{D}/z$. In this negative term we use $\log(\sqrt{D}/d) \leq (1 - 1/s)^{-1} \log \sqrt{D}$, obtaining
\[
 \{ \cdots \}
 \geq \bigg( \log \frac{\sqrt{D}}{d} \bigg)^{2l+2}
  - \,\,\,\, \frac{2s}{l+2} \bigg(1 - \frac{1}{s}\bigg)^{-l-1} \big( \log \sqrt{D}\big)^l \bigg( \log \frac{\sqrt{D}}{zd}\bigg)^{l+2}
\]
with the negative term being present only for $d < \sqrt{D}/z$. Inserting 
this into~\eqref{eYsGd'LB} we get two sums, over $d < \sqrt{D}$ in the 
positive terms and over $d < \sqrt{D} / z$ in the negative terms. The 
positive sum is of type~\eqref{e241SlDef} with $\kappa = k - 1$ and $l$ 
replaced by $2l + 2$. The negative sum is also of type~\eqref{e241SlDef} 
with the same $\kappa = k-1$, but with $l$ replaced by $l+2$. Let $K(s)$ 
and $L(s)$ denote the corresponding functions defined 
in~\eqref{e250DifDefHBdry} and~\eqref{e251DifDifH}. 
Applying~\eqref{e249SlxzBySlxx} and comparing the resulting lower 
bound for $Y_s^2(D) G_s'(D)$ with the asymptotic value of $Y_1^2(D) G_1'(D)$ 
we conclude that
\[
 G_s'(D)
 > \chi(s) G'(D) \big\{
   1 + O ( 1 / \log D )
   \big\},
\]
where
\[
 \chi(s)
 = K(s)
  - \frac{2s}{l+2} \left( 1 - \frac{1}{s}\right)^{-l-1} \left( 1 - \frac{1}{s}\right)^{k+l+1} L(s)
\]
Here, the factor $(1 - 1/s)^{k+l+1}$ in front of $L(s)$ comes by way 
of re-scaling:
\[
 Y_1^2(D/z^2) G_1'(D/z^2)
 = Y_1^2(D) G_1'(D)
 \left( \frac{\log \sqrt{D} / z}{\log \sqrt{D}} \right)^{k + l + 1}
  \left\{ 1 + O \left(\frac{1}{\log D}\right)\right\}.
\]
Finally, by the lower bound~\eqref{e252HLowBd} for $K(s)$ and the upper bound $L(s) \leq 1$, we find that
\[
 \chi(s)
 > 1 - \frac{(k-1)s}{k + 2l + 2} \left( 1 - \frac{1}{s} \right)^{k + 2l + 2}
  -\,\, \frac{2s}{l+2} \left(1 - \frac{1}{s}\right)^k
 > 1 - 2s \left( 1 - \frac{1}{s} \right)^k.
\]
This completes the proof of~\eqref{e237Gs'LB} ($D$ is assumed to be large).

\section{\bf{Proof of the Main Theorem}}
\label{c3proofMainThm}

\subsection{Sums over primes}

The most common technique for handling sums over primes uses combinatorial formulas of sieve exclusion-inclusion type; see Chapters 17,18 of~\cite{FI2}, where many variations on this theme are illustrated with applications. In this paper we develop a smooth version of Heath-Brown's formula.

Let $\psi(u)$ be a smooth function supported on $|u| \leq 2$ with
\begin{equation}
 \psi(u) = 1
  \quad \text{if } |u| \leq 1.
\end{equation}
Let $K \geq 1$ and $M \geq 1$. For any integer $n$ with $1 \leq n \leq 2M^K$ we have
\begin{multline}
 \Lambda(n)
 = -\sum_{1 \leq J \leq K} (-1)^J \binom{K}{J}
  \mathop{\sum \cdots \sum}_{l_1 \cdots l_J m_1 \cdots m_J = n}
  \mu(m_1) \cdots \mu(m_J)
\\
  \cdot \psi\left(\frac{m_1}{M}\right) \cdots \psi\left(\frac{m_J}{M}\right) \log l_1.
\label{e32LambdaByMultPsi}
\end{multline}
This can be seen from the following lines:
\begin{multline*}
 \frac{\zeta'}{\zeta}(s)
 \left( 1 - \zeta(s) \sum_m \mu(m) \psi\left(\frac{m}{M}\right) m^{-s} \right)^K
 = \sum_{n > 2M^K} a_n n^{-s}
\\
 = \frac{\zeta'}{\zeta}(s)
  + \sum_{1 \leq J \leq K} (-1)^J \binom{K}{J} \zeta'(s) \zeta(s)^{J-1} \left(
   \sum_m \mu(m) \psi\left(\frac{m}{M}\right) m^{-s}
   \right)^J.
\end{multline*}

We shall sum arithmetic functions over primes with smooth weights. For technical convenience the sum is often split into dyadic boxes. Here is how one can perform such a split in a general context without compromising the smoothness. Suppose $F = \phi * \cdots * \phi * \Phi$ is given as the multiplicative convolution:
\begin{equation}
 F(x)
 = \idotsint \phi(t_1) \cdots \phi(t_r) \Phi(t_1 \cdots t_r / x) \dint^* t,
\label{e33ConvDef}
\end{equation}
where $\dint^* t = (t_1 \cdots t_r)^{-1} \dint t_1 \cdots \dint t_r$. Choose 
$\phi(y)$ and $\Phi(y)$, smooth and supported on $1 \leq y \leq 2$, with
\begin{equation}
 \int \phi(y) y^{-1} \dint y
 = 1.
\label{e34PhiNormInt1}
\end{equation}
Let $N \geq 1$. Note that if $1 \leq n \leq 2^r N$, then
\[
 \int_{1/2}^{2^r N} \phi\left( \frac{n}{y} \right) \frac{\dint y}{y}
 = \int \phi(y) y^{-1} \dint y
 = 1.
\]
Hence, for any numbers $n_1,\dots,n_s \geq 1$, $1 \leq s \leq r$, we have
\begin{multline}
 F\left( \frac{n_1\cdots n_s}{N}\right)
 = \idotsint \int_{1/2}^{2^r N} \cdots \int_{1/2}^{2^r N}
  \phi\left( t_1 \frac{n_1}{y_1} \right)
  \phi\left( \frac{n_1}{y_2}\right) \cdots \phi\left(t_s \frac{n_s}{y_s}\right) \phi \left( \frac{n_s}{y_s}\right)
\\
 \cdot \phi(t_{s+1})\cdots \phi(t_r) \Phi\left( \frac{t_1 \cdots t_r}{y_1 \cdots y_s} N \right) \dint^* t \dint^* y.
\label{e35FMultConv}
\end{multline}
This integral representation provides simultaneously a separation of the 
variables $n_1$, \dots, $n_r$ and their partition into smooth, dyadic-type 
boxes. Indeed, the integration variables satisfy
\begin{align}
 \frac{1}{2} 
 &< t_j < 2\ , \quad\quad\quad {\rm for}\quad 
1 \leq j \leq r,
\label{e34IntDyBoxtj}
\\
 \frac{1}{2}
 &< y_j < 2^r N\ , \quad\quad {\rm for}\quad 
1 \leq j \leq s,  
\\
 2^{-r-1} N
 &< y_1 \cdots y_s < 2^r N,
\label{e36IntDyBoxProdY}
\end{align}
while the projected summation variables $n_1,\dots,n_s$ are well located in the segments
\begin{equation}
 y_j < n_j < 2y_j,
  \qquad 1 \leq j \leq s.
\end{equation}
Note that the measure of the set~\eqref{e34IntDyBoxtj}-\eqref{e36IntDyBoxProdY} is bounded by
\begin{equation}
 (\log 4)^r ( \log 2^{r+1} N)^{s-1} \log 2^{2r+1}
 \ll  \left( 3r/2 \right)^r (\log 2N)^{s-1}.
\end{equation}
Of course, one can take the convolution $F = \phi_1 * \cdots * \phi_r * \Phi$ with a different test function for every variable, however, our choice $\phi_1 = \cdots = \phi_r = \phi$ does the job fine.

We shall use the integral~\eqref{e35FMultConv} with $r = 2K$ and $s = 2J$, 
so for reference we denote the domain of 
integration~\eqref{e34IntDyBoxtj}-\eqref{e36IntDyBoxProdY} 
by $\cM_{JK}(N)$. The measure of $\cM_{JK}(N)$ (with respect to the 
multiplicative group) satisfies
\begin{equation}
 | \cM_{JK}(N)|
 \leq 2 (3K)^{2K} (\log 2N)^{2J-1}.
\end{equation}

Applying~\eqref{e35FMultConv} to~\eqref{e32LambdaByMultPsi} we reach  
the following ``decomposition'' formula for smoothly cropped prime 
powers.

\begin{prop}
Let $F = \phi * \cdots * \phi * \Phi$ be the multiplicative convolution 
of $2K$ copies of a smooth function $\phi(y)$ supported in $1 \leq y \leq 2$ 
and normalized by~\eqref{e34PhiNormInt1}, and $\Phi(y)$ supported in 
$1 \leq y \leq 2$. Let
\begin{equation}
 1 \leq N \leq 2(M/4)^K.
\end{equation}
Then, for any $n \geq 1$ we have
\begin{multline}
 \Lambda(n) F\left(\frac{n}{N}\right)
 =
 - \sum_{1 \leq J \leq K} (-1)^J \binom{K}{J}
  \int_{\cM_{JK}(N)} \bigg(
   \prod_{2 J < j \leq 2K} \phi(t_j)
   \bigg)
  \Phi\left( \frac{t_1 \cdots t_{2K}}{y_1 \cdots y_{2J}} N \right)
\\
  A_n(t_1,\dots,t_{2J}; y_1,\dots,y_{2J}) \frac{\dint t_1 \cdots \dint t_{2K}}
{t_1 \cdots t_{2K}}\frac{\dint y_1 \cdots \dint y_{2J}}{y_1 \cdots y_{2J}}
\label{e314LambdaDecompConv}
\end{multline}
where $A_n(\,;)$ is given by the Dirichlet convolution of $2J$ factors
\begin{multline}
 A_n(\,;)
 = \mathop{\sum \cdots \sum}_{l_1\cdots l_J m_1 \cdots m_J = n}
  \mu(m_1) \cdots \mu(m_J) \psi\left(\frac{m_1}{M}\right) 
\cdots \psi\left(\frac{m_J}{M}\right)
\\
 (\log l_1) \prod_{1 \leq j \leq J}
  \phi\left(\frac{l_j}{y_j}\right)
  \phi\left(\frac{l_j}{y_j} t_j\right)
  \phi\left(\frac{m_j}{y_{J+j}}\right)
  \phi\left(\frac{m_j}{y_{J+j}} t_{J+j}\right).
\label{eseqA}
\end{multline}
\end{prop}

\begin{rems}
The function $F(x)$ is supported in the segment 
$\frac{1}{2} \leq x \leq 2^{2K}$, therefore $\Lambda(n) F(n/N)$ captures 
the prime powers inside $\frac{N}{2} < n < 2^{2K} N$ with weights smoothly 
vanishing at the endpoints.

The Fourier transform $\hat{F}(s)$ at $s=0$ will appear as a factor in the 
main terms. For $F = \phi * \cdots * \phi * \Phi$ given 
by~\eqref{e33ConvDef} we have
\[
 \hat{F}(0)
 = \left( \int \phi(y) \dint y \right)^r \int \Phi(y) y^{-2} \dint y.
\]
Given $\phi$ and $r \geq 1$ we shall take $\Phi$ to be normalized  so that
\begin{equation}
 \hat{F}(0)
 = 1.
\end{equation}
\end{rems}

\subsection{Combinatorial arrangements}

For clarity of exposition, let us arrange the 
ordinates of the points in the domain of 
integration $\cM_{JK}(N)$ 
(see~\eqref{e34IntDyBoxtj}-\eqref{e36IntDyBoxProdY})  
into the decreasing sequence 
$\{y_1,y_2,\dots,y_{2J}\}
=\{Y_1,Y_2,\dots,Y_{2J}\}$, say, with 
\begin{equation}
 Y_1 \geq Y_2 \geq \cdots \geq Y_{2J}.
\label{e317Ydec}
\end{equation}
By~\eqref{e36IntDyBoxProdY} we get
\begin{equation}
 Y_1 Y_2 \cdots Y_{2J} \asymp N.
\label{e318Yprod}
\end{equation}

\begin{lem} \label{l32Yprods}
Let $0 < \eta < 1/40$. Then, at least one of the following three cases hold:
\begin{enumerate}[label={\normalfont{(C\arabic*)}}]
\item\label{C1} we have
\[
 Y_1 \gg N^{\frac{5}{8} - \eta} , 
\]
\item\label{C2} there is a subproduct of~\eqref{e318Yprod}, say $X$, with
\[
 N^{\frac{3}{8} + \eta}
 \ll X
 \ll N^{\frac{5}{8} - \eta},
\]
\item\label{C3} we have
\[
 Y_1
 \geq Y_2
 \geq Y_3
 \gg N^{\frac{1}{4} - \eta}
\quad \text{and} \quad
 Y_1 Y_2 Y_3^{\frac{5}{4}}
 \gg N^{\frac{65}{64} - \frac{13}{8} \eta}.
\]
\end{enumerate}
The implied constants in~\ref{C1}, \ref{C2}, \ref{C3} depend only on the implied constant in~\eqref{e318Yprod}.
\end{lem}

Obviously, Lemma~\ref{l32Yprods} follows from the next lemma, which asserts similar but neat inequalities.

\begin{lem}
Let $0 < \eta < 1/40$ and $\nu_1,\dots,\nu_r$ be positive numbers with
\begin{gather}
 \nu_1 \geq \dots \geq \nu_r
\tag{\ref{e317Ydec}'}
\\
 \nu_1 + \cdots + \nu_r = 1.
\tag{\ref{e318Yprod}'} \label{e318'NuSum}
\end{gather}
Then, at least one of the following three cases hold:
\begin{enumerate}[label={\normalfont{(C\arabic*')}}]
\item\label{C1'} we have
\[
 \nu_1
 > \frac{5}{8} - \eta,
\]
\item\label{C2'} there is a subsum of~\eqref{e318'NuSum}, say $\mu$, with
\[
 \frac{3}{8} + \eta
 < \mu
 < \frac{5}{8} - \eta,
\]
\item\label{C3'} we have
\[
 \nu_1 \geq \nu_2 \geq \nu_3 \geq \frac{1}{4} - 2 \eta
  \quad \text{and} \quad
 \nu_1 + \nu_2 + \frac{5}{4} \nu_3
 \geq \frac{65}{64} - \frac{13}{8} \eta.
\]
\end{enumerate}
\end{lem}

\begin{proof}
Negating~\ref{C1'}, \ref{C2'}, we are going to derive~\ref{C3'}. First we find that $\nu_1 \leq \frac{3}{8} +\eta$. Hence $r \geq 3$. Let $j \geq 2$ be the largest such that
\[
 \mu
 = \nu_2 + \cdots + \nu_j
 < \frac{5}{8} - \eta.
\]
Since $\nu_2 + \cdots + \nu_r = 1 - \nu_1 \geq \frac{5}{8} - \eta$, it follows that $2 \leq j < r$. Then, by the negation of~\ref{C2'}, we have $\mu \leq \frac{3}{8} + \eta$. Since $\mu + \nu_{j+1} \geq \frac{5}{8} - \eta$, we get
\[
 \nu_3
 \geq \nu_{j+1}
 \geq \frac{1}{4} - 2 \eta.
\]
Hence we get $\nu_2 + \nu_3 \geq 2 \nu_3 \geq \frac{1}{2} - 4 \eta > \frac{3}{8} + \eta$, and by the negation of~\ref{C2'} this lower bound improves to $\nu_2 + \nu_3 \geq \frac{5}{8} - \eta$. Hence,
\[
 \nu_1 + \nu_2 + \frac{5}{4} \nu_3
 \geq 2 \nu_2 + \frac{5}{4} \nu_3
 \geq \frac{13}{8} ( \nu_2 + \nu_3 )
 \geq \frac{65}{64} - \frac{13}{8} \eta,
\]
so the conditions~\ref{C3'} hold.
\end{proof}


\subsection{Primes in residue classes to divisor-rich moduli}

Given $z \geq 2$ we say that a squarefree number $q$ is \emph{divisor-rich} if it has divisors of any size up to a factor $\leq z$. For example, if $q\mid P(z)$ then $q$ is divisor-rich. 

Let $Z(X) \in \Z[X]$ be a polynomial which has no fixed prime divisors, that is, for every prime $p$ the number of roots of $Z(x) \equiv 0 \pmod{p}$ is $\omega(p) < p$. The following result is a refined version of Theorem 2 of Zhang~\cite{Z}.

\begin{prop} \label{p34LambdaSumRichEst}
Let $\cQ = \cQ(z)$ be the set of divisor-rich numbers.
We have
\begin{equation}
 \sum_{\substack{q \in \cQ \\ q \leq Q}} \sumst_{\substack{a (\mathrm{mod} \, q) \\ q \mid Z(a)}} \left|
  \sum_{n \equiv a \,(q)} \Lambda(n) F\left(\frac{n}{N} \right) - \frac{N}{\varphi(q)}
  \right|
 \ll N (\log N)^{-A},
\label{e319LambdaSumRich}
\end{equation}
provided
\begin{equation}
 z^{\frac{86}{207}} Q
 \leq N^{\frac{104}{207} - \eps}.
\label{e320condZQN}
\end{equation}
Here $\eps,A$ are any positive numbers and the implied constant depends only on $\eps$, $A$, and the polynomial $Z(X)$.
\end{prop}

\begin{rems}
One can replace the smooth cropping $F(n/N)$ by the sharp cutting $n \leq N$, 
but this is not required in applications. However, the restriction to 
divisor-rich 
moduli is essential in Zhang's arguments. 
\end{rems}

The proof of Proposition~\ref{p34LambdaSumRichEst} depends on several 
results on the distribution of general sequences over residue classes 
which are established in Sections~\ref{c4bilForms} to~\ref{c6CroppDiv}. In 
this section we derive~\eqref{e319LambdaSumRich} by combining these results.

The sum over $a$ in ~\eqref{e319LambdaSumRich} makes the result look, at 
first glance, deceptively strong, but actually, there are not so many residue 
classes involved; $\omega (q)\ll \tau(q)^k$. 
Hence, it suffices to show that 
\begin{equation}
 \sum_{\substack{q \in \cQ \\ q \leq Q}} \frac{1}{\omega(q)}
\sumst_{\substack{a (\mathrm{mod} \, q) \\ q \mid Z(a)}} \left|
  \sum_{n \equiv a \, (q)} \Lambda(n) F\left(\frac{n}{N}\right) 
- \frac{N}{\varphi(q)}
  \right|
 \ll N (\log N)^{-A},
\label{eSumTauLambdaFBd}
\end{equation}
Clearly, ~\eqref{eSumTauLambdaFBd}
implies Proposition~\ref{p34LambdaSumRichEst}, by 
applying Cauchy's inequality and the easy bound
\[
 \sum_{q \leq Q} \tau(q)^{2k} \Bigg(
  \sum_{\substack{n \equiv a \,(q) \\ n \leq 2N}} \Lambda(n) + \frac{N}{\varphi(q)}
  \Bigg)
 \ll N (\log N)^{4^k}.
\]
Moreover, the Bombieri-Vinogradov Theorem (see~\eqref{e16BomVinDistMod} 
or (9.81) of~\cite{FI2}) covers the range $q \leq N^{\frac{1}{2} - \frac{\eps}{4}}$.

Now, for the proof of~\eqref{eSumTauLambdaFBd} we can assume that $N$ is 
sufficiently large in terms of $\eps$, and that $q$ runs over $\cQ=\cQ(z)$ in 
the segment $Q < q \leq 2Q$ with
\begin{equation}
 Q
 \geq N^{\frac{1}{2} - \frac{\eps}{3}}
\end{equation}
Therefore~\eqref{e320condZQN} implies
\begin{equation}
 z
 \leq N^{1 / 172}.
\label{e322ZleqN172}
\end{equation}
We are going to apply the decomposition~\eqref{e314LambdaDecompConv} with $K = 5$ and $M = 4N^{\frac{1}{5}}$. It suffices to show that
\begin{equation}
 \sum_{\substack{q \sim Q\\q \in \cQ}} \frac{1}{\omega(q)}
\sumst_{\substack{a (\mathrm{mod} \, q) \\ q \mid Z(a)}} \left|
  \sum_{n \equiv a\,(q)} A_n - \frac{1}{\varphi(q)} \sum_{(n,q) = 1} A_n
  \right|
 \ll N(\log N)^{-A}
\label{e323SumAnEst}
\end{equation}
for every sequence $A_n = A_n(t_1,\dots,t_{2J};y_1,\dots,y_{2J})$ 
of type~\eqref{eseqA}. Recall that
\begin{align}
 \frac{1}{2} 
 &< y_j < 2^{10} N,
 & & 1 \leq j \leq 2J,
\\
 2^{-11} N
 &< y_1 \cdots y_{2J}
 < 2^{10}N,
\end{align}
and $\frac{1}{2} < t_j < 2$, $1 \leq j \leq 2J$, with $1 \leq J \leq 5$. Let $\eta$ be given by
\begin{equation}
 z^{\frac{3}{2}} Q^{\frac{11}{4}}
 = N^{\frac{11}{8} + \eta - \eps}.
\label{e326EtaDef}
\end{equation}
By~\eqref{e320condZQN} and~\eqref{e322ZleqN172} one can check that $\eps / 12 \leq \eta < 1/114$, a fortiori $0 < \eta < 1/40$. According to Lemma~\ref{l32Yprods} we can arrange the set $\{y_1,y_2,\dots,y_{2J}\}$ into a decreasing sequence $Y_1 \geq Y_2 \geq \dots \geq Y_{2J} > \frac{1}{2}$ which satisfies at least one of the three conditions~\ref{C1}, \ref{C2}, \ref{C3}.

If~\ref{C1} holds, then $Y_1 \gg N^{3/5}$, so the corresponding variable 
$l_j\asymp Y_1$ appears in~$A_n, n=l_1\ldots l_Jm_1\ldots m_J$, with a smooth crop function and the inner sum in~\eqref{e323SumAnEst} is bounded by $N Y_1^{-1} \ll N^{2/5}$. Summing over $q$ we see that in this case~\eqref{e323SumAnEst} is bounded by
\[
 N^{\frac{2}{5}} Q
 \ll N^{\frac{9}{10} + \frac{1}{414}}
 \ll N ( \log N)^{-A}.
\]

Suppose~\ref{C2} holds, so the sequence $Y_1 \geq Y_2 \geq \dots \geq Y_{2J}$ 
can be partitioned into two subsequences whose products, say $N_1,N_2$, satisfy
\[
 N^{\frac{3}{8} + \eta}
 \ll N_1,N_2
 \ll N^{\frac{5}{8} - \eta},
  \quad N_1 N_2 \asymp N.
\]
Obviously, each of the two subsequences has an element, $Y'$ and $Y''$ 
respectively, with $Y',Y'' \gg N^{3/80}$ and the corresponding variables 
$m'\asymp Y', m''\asymp Y''$ appear in $A_n$ with a smooth crop 
function multiplied, or not, 
by the M\"obius function. Therefore, the S-W condition holds for each of 
the two convolution sequences, 
see the examples~\eqref{e4.7.1} and~\eqref{e4.7.2}. We can assume, 
by exchanging $N_1$ and $N_2$ if necessary, that
\begin{equation}
 N^{\frac{3}{8} + \eta}
 \ll N_1
 \ll N^{\frac{1}{2}}.
\label{e327N1betN}
\end{equation}
Note that $Q \gg N_1 N^{-\eps / 3}$. We can factor the divisor-rich moduli 
$q \sim Q$ into $q = rs$ with $r \asymp R$, $s \asymp S$, $RS = Q$, such that
\begin{equation}
 z^{-1} N_1
 \ll R N^\eps
 \ll N_1.
\end{equation}
Then~\eqref{e323SumAnEst} is bounded by 
$$
\sum_r\max_{(a,r)=1}\sum_s\max_{(b,s)=1}
 \left|
  \sum_{\substack{n \equiv a\,(r)\\n \equiv b\,(s)}} A_n - \frac{1}{\varphi(rs)} \sum_{(n,rs) = 1} A_n
  \right|
$$
For $\eta$ given by~\eqref{e326EtaDef}, the lower bound of~\eqref{e327N1betN} 
becomes
\begin{equation}
 z^{\frac{3}{2}} Q^{\frac{11}{4}}
 \ll N_1 N^{1-\eps}.
\end{equation}
The above two estimates verify the conditions of Corollary~\ref{corEstBilForm},
with $x$, $N$ replaced by $N$, $N_1$, respectively.
Therefore,~\eqref{eEstBilForm} of Corollary~\ref{corEstBilForm} 
applies, showing 
that~\eqref{e323SumAnEst} is bounded by $N(\log N)^{-A}$, as requested.

Now we are left with the third condition~\ref{C3}. In this case we have $Y_1 \geq Y_2 \geq Y_3 \geq 8N^{\frac{1}{5}} = 2M$, so the corresponding variables $n_1 \asymp Y_1$, $n_2 \asymp Y_2$, $n_3 \asymp Y_3$, appear in $A_n$ with smooth crop functions. The product of the remaining variables is $d = n / n_1 n_2 n_3$, so $d Y_1 Y_2 Y_3 \asymp N$. For $\eta$ given by~\eqref{e326EtaDef}, the second estimate of~\ref{C3} becomes
\begin{equation}
 Y_1 Y_2 Y_3^{\frac{5}{4}}
 \gg N^{\frac{65}{64}} ( z^{\frac{3}{2}} Q^{\frac{11}{4}} N^{\eps - \frac{11}{8}} )^{-13/8}.
\end{equation}
Hence, using~\eqref{e320condZQN} one can derive the following upper bound for $Q$
\begin{equation}
 Q
 \ll \left( \frac{Y_3}{z} \right)^{\frac{1}{8}} \left( \frac{N}{d} \right)^{\frac{1}{2}} N^{-\frac{155}{64} \eps}.
\end{equation}
This verifies the condition~\eqref{esixten} of Corollary~\ref{cor53estSumFlmn},
with $L$, $M$, $N$, $x$ replaced by $Y_1$, $Y_2$, $Y_3$, $N$ respectively,   
which proves~\eqref{e323SumAnEst} in this case.

Having covered all cases, we complete the proof of 
~\eqref{eSumTauLambdaFBd} and Proposition~\ref{p34LambdaSumRichEst}.


\subsection{Completing the proofs of the theorems}
\label{s34ProofsThms}

Recall the formula~\eqref{e229Wexpw} in which $w(D,N)$ satisfies the lower bound \eqref{e238wLB}. Here $D = N^\theta$ is the level of distribution for primes in residue classes which must ensure that the remainder terms in~\eqref{e229Wexpw} are negligible, meaning
\begin{equation}
 R_h(D,N)
 \ll N(\log N)^{-A}.
\label{e333RhBd}
\end{equation} 
Recall that $z = D^{1/2s}$ is the level of the sieve support restriction. 
Proposition~\ref{p34LambdaSumRichEst} ensures~\eqref{e333RhBd}, provided
\begin{equation}
 z^{\frac{86}{207}} D
 \leq N^{\frac{104}{207} - \eps}.
\label{e334zDBd}
\end{equation}
Putting $D = N^{\frac{1}{2} + \delta}$ we definitely get~\eqref{e334zDBd} if
\begin{equation}
 \delta
 < \frac{1}{144} \left( 1 - \frac{43}{s} \right).
\label{e335DeltaCond}
\end{equation}
On the other hand, the positivity condition~\eqref{e29CkPos} requires
\begin{equation}
 \delta
 > \frac{1}{\sqrt{k}}
  + \frac{1}{k}
  + (1 + 2\delta) s e^{-k/s}.
\label{e336DeltaLB}
\end{equation}
We choose $s = 4494$ and find that $\delta = 1/418$ satisfies~\eqref{e335DeltaCond}. Then we find that $k = (419)^2 = 175561$ satisfies~\eqref{e336DeltaLB}. This completes the proof of~\eqref{e27WAsymp}.

Applying the inequality
\[
 \sum_j c_j
 \leq 1 + \mathop{\sum \sum}_{i < j} c_i c_j,
\]
which holds for any $c_j$ with $0 \leq c_j \leq 1$, we derive by~\eqref{e25defan}, \eqref{e26WeLB}, \eqref{e27WAsymp} that
\begin{equation}
 \sum_{(Q(n),\Delta)=1} F(n/N) \mathop{\sum \sum}_{h \neq h'}
 \Lambda(n-h) \Lambda(n-h') \Bigg(
  \sum_{\substack{d \mid Q(n) \\ d \mid P(z)}} \lambda_d
  \Bigg)
 \gg N (\log N)^{2 - k}.
\label{e338LambdanhLB}
\end{equation}
Here the contribution of $n$'s with $Q(n)$ having a very 
small prime divisor is relatively insignificant, 
see Section~10.3 of~\cite{FI2}. Specifically, by~(10.7) of~\cite{FI2} we get
\[
 \sum_{(Q(n),\Delta)=1} F(n/N) \Bigg(
  \sum_{\substack{p \mid Q(n) \\ p < w}} 1
  \Bigg)
 \Bigg(
  \sum_{\substack{d \mid Q(n) \\ d \mid P(z)}} \lambda_d
  \Bigg)
 \ll \frac{\log w}{\log N} N (\log N)^{-k}
\]
for every $w \geq 2$. Hence, if $w = N^{1/b}$ with $b$ sufficiently large, this contribution can be ignored and~\eqref{e338LambdanhLB} holds for the sum with the extra restriction $(Q(n),P(w))=1$. Having this extra restriction, we can remove the sieve weights by the following trivial estimation:
\[
 \sum_{\substack{d \mid Q(n) \\ d \mid P(z)}} \lambda_d
 \leq \sum_{d \mid Q(n)} \tau_3(d)
 \leq \tau_4\big(Q(n)\big)
 \leq 4^{b k}.
\]
This completes the proof of the lower bound of~\eqref{e28MainThm}. To get an 
upper bound we replace $\Lambda(n-h)\Lambda(n-h')$ by $(\log n)^2$ and apply 
any reasonable sieve which controls the condition $(Q(n),P(w))=1$. This 
completes the proof of the Main Theorem.

For the proof of the Prime Gaps Theorem we require a specific admissible set 
$\cH = \{h_1,\dots,h_k\}$. One can easily see that any set of $k$ distinct 
primes larger than $k$ is admissible. We choose the set of consecutive 
primes $k < h_1 < h_2 < \dots < h_k$, where $h_1$ is the first exceeding 
$k$ and $h_k$ is the largest of the bunch. This satisfies
\begin{equation}
 \pi(h_k)
 = k + \pi(k),
\end{equation}
so $h_k \sim k \log k$. For $k = 175561$ we find the following numbers:
\begin{equation*}
\begin{gathered}
 h_1
 = 175573
 = p_{15954},
\\
 \pi(h_k)
 = k + \pi(k)
 = 175561 + 15953
 = 191 \, 514,
\\
 h_k
 = p_{191514}
 = 2 \, 624 \, 371,
\\
 h_k - h_1
 = 2 \, 448 \, 798.
\end{gathered}
\end{equation*}
This completes the proof of the Prime Gaps Theorem.

\section{\bf{Bilinear Forms in Residue Classes}}
\label{c4bilForms}

\subsection{Classical results}

Counting prime numbers in residue classes $a \pmod{q}$ with $(a,q) = 1$ has 
been transformed in Section~\ref{c3proofMainThm} to estimates for bilinear 
forms of type
\begin{equation}
 \mathop{\sum \sum}_{mn \equiv a \, (q)} \alpha_m \beta_n
\label{e41bilForm}
\end{equation}
for special sequences $\cA = (\alpha_m)$, $\cB = (\beta_n)$ with
\begin{align}
 |\alpha_m|
 &\leq \tau(m)^j ,
\label{e42BilFormAlpBd}
\\
 |\beta_n|
 &\leq \tau(n)^j . 
\label{e43BilFormBetaBd}
\end{align}
Here and hereafter, $j$ denotes a bounded number which can be different 
on each
occurrence. The available technology is capable of dealing with such 
bilinear forms in great generality if at least one of the sequences is 
known to be uniformly distributed over the reduced residue classes to 
moduli which are relatively small. Specifically, we shall assume and 
appeal several times to the following properties of $(\beta_n)$.

\begin{swcond}
Let $\cB = (\beta_n)$ be a sequence of complex numbers with $|\beta_n| \leq \tau(n)^j$ for some $j \geq 1$. We say that $\cB$ satisfies the \emph{Siegel-Walfisz condition} in the segment $1 \leq n \leq N$ if, for every $k \geq 1$ and $(l,k)=1$, we have
\begin{equation}
 \sum_{\substack{n \leq N \\ n \equiv l \, (k)}} \beta_n
 - \frac{1}{\varphi(k)} \sum_{\substack{n \leq N \\ (n,k) = 1}} \beta_n
 \ll N (\log 2N)^{-A}
\label{e44SWcond}
\end{equation}
with any $A \geq 1$, the implied constant depending only on $A$ and $j$.
\end{swcond}

\begin{prop}
\label{p41BetaSumRes}
If $(\beta_n)$ satisfies the S-W condition in the segment $1 \leq n \leq N$, 
then
\begin{equation}
 \sum_{q \leq Q} \sumst_{a \, (q)} \Bigg|
  \sum_{\substack{n \leq N \\ n \equiv a \, (q)}} \beta_n
  - \frac{1}{\varphi(q)} \sum_{\substack{n \leq N \\ (n,q) = 1}} \beta_n
  \Bigg|^2
 \ll \big(Q + N (\log 2N)^{-A} \big) N (\log 2N)^{2^j}
\label{e45BetaSumRes}
\end{equation}
for any $Q \geq 1$ and any $A \geq 1$ the implied constant depending only 
on $A$ and $j$.
\end{prop}

Proposition~\ref{p41BetaSumRes} follows from Theorem 9.14 of~\cite{FI2}; 
see the end of its proof on p. 168 to get the above statement without any 
restriction on $Q$. Of course, the best use of~\eqref{e45BetaSumRes} is 
for $Q \ll N (\log 2N)^{-A}$.

In practice we often need~\eqref{e45BetaSumRes} for the subsequence of
 $(\beta_n)$ restricted by the co-primality condition $(n,r) = 1$. This 
can be deduced easily at the expense of losing a factor $\tau(r)^2$ in 
the bound~\eqref{e45BetaSumRes}. To this end we verify~\eqref{e44SWcond} 
for the subsequence by using the Möbius formula twice as follows:
\begin{align*}
 \sum_{(n,r)=1} \xi_n
 &= \sum_{\substack{d \mid r \\ d \leq y}} \mu(d) \sum_{n \equiv 0 \, (d)} \xi_n
   + O \Bigg( \sum_{\substack{d \mid r \\ d > y}} 
\left( \frac{N}{d} \sum_n |\xi_n|^2 \right)^{\frac{1}{2}} \Bigg)
\\
 &= \sum_{\substack{bc \mid r \\ bc \leq y}} \mu(b) \sum_{(n,b) = 1} \xi_n
   + O \Bigg( \tau(r) \left( \frac{N}{y} 
\sum_n | \xi_n |^2 \right)^{\frac{1}{2}} \Bigg).
\end{align*}
Here and below we assume that $n \leq N$. Hence, for $(r,k) = 1$ we get
\begin{multline*}
 \Delta(r)
 = \sum_{\substack{(n,r) = 1 \\ n \equiv l \, (k)}} \beta_n
  - \frac{1}{\varphi(k)} \sum_{(n,rk) = 1} \beta_n
 = \sum_{\substack{bc \mid r \\ bc \leq y}} \mu(b) \Bigg(
  \sum_{\substack{(n,b) = 1 \\ n \equiv l \, (k)}} \beta_n
  - \frac{1}{\varphi(k)} \sum_{(n,bk) = 1} \beta_n
  \Bigg)
\\
  + O \Big( \tau(r) y^{-\frac{1}{2}} N (\log 2N)^{2^j} \Big).
\end{multline*}
Splitting into primitive classes $n \equiv \beta \pmod{b}$ and 
applying~\eqref{e44SWcond} with $k$ replaced by $bk$ we get the bound
\begin{align*}
 \Delta(r)
 &\ll \sum_{\substack{bc \mid r \\ bc \leq y}} \varphi(b) N ( \log 2N )^{-4A}
   + \tau(r) y^{-\frac{1}{2}} N (\log 2N)^{2^j}
\\
 &\ll \tau(r) N \big( y (\log 2N)^{-4A} + y^{-\frac{1}{2}} (\log 2N)^{2^j}\big).
\end{align*}
Of course, the assumption $(r,k) = 1$ can be ignored by applying the result for the maximal divisor of $r$ which is co-prime with $k$. This leads to

\begin{cor} \label{cor411estBetaSW}
Suppose $\cB = ( \beta_n)$, with $|\beta_n| \leq \tau(n)^j$, satisfies the S-W condition in the segment $1 \leq n \leq N$. Then for any $r \geq 1$, $k \geq 1$, and $(l,k)=1$, we have
\begin{equation}
 \sum_{\substack{n \leq N \\ n \equiv l \, (k) \\ (n,r)=1}} \beta_n
  - \frac{1}{\varphi(k)} \sum_{\substack{n \leq N \\ (n,r)=1}} \beta_n
 \ll \tau(r) N (\log 2N)^{-A}
\end{equation}
with any $A \geq 1$, the implied constant depending only on $A$ and $j$.
\end{cor}

Hence, indeed,~\eqref{e45BetaSumRes} holds for the sum restricted by $(n,r)=1$ with an extra factor $\tau(r)^2$ on the right side. 

To ease applications, we also introduce into~\eqref{e45BetaSumRes} the factor 
$\tau(q)^j$ at a minor expense. Specifically, applying Cauchy's inequality 
and the direct bound 
\[
 \sum_{q \leq Q} \tau(q)^{2j} \sumst_{a \, (q)} \Bigg( 
 \sum_{\substack{n \leq N \\ n \equiv a\,(q)}} \tau(n)^j
  + \frac{1}{\varphi(q)} \sum_{\substack{n \leq N \\ (n,q) = 1}} \tau(n)^j
  \Bigg)^2
\ll (Q+N)(\log 2QN)^{4^{j+1} }
\]

together with~\eqref{e45BetaSumRes}, we derive the following more practical 
result: 
\begin{cor} \label{cor42BdBetaSW}
If $(\beta_n)$ satisfies the S-W condition in the segment $1 \leq n \leq N$, 
then
\begin{align*}
 \sum_{q \leq Q} \tau(q)^j \sumst_{a \, (q)} & \Bigg| 
  \sum_{\substack{n \leq N \\ n \equiv a\,(q) \\ (n,r) = 1}} \beta_n
  - \frac{1}{\varphi(q)} \sum_{\substack{n \leq N \\ (n,qr) = 1}} \beta_n
  \Bigg|^2 \\
&  \ll \tau(r) \big( Q + \sqrt{QN} + N (\log 2N)^{-A} \big) 
N (\log 2QN)^{4^{j+1} }
\label{e4.7}
\end{align*}
for any positive integer $r$, any $Q \geq 1$, and any $A \geq 1$, the implied constant depending only on $A$ and $j$.
\end{cor}

\begin{examps}
Obviously, the S-W condition holds for $\beta_n = f(n/N)$ where $f(x)$ is a continuous function compactly supported on $\R^+$. By the Prime Number Theorem the S-W condition holds for
\begin{equation}
 \beta_n
 = \mu(n) f(n/N).
\label{e4.7.1}
\end{equation}
The S-W condition is pretty flexible. Suppose that $\cB = (\beta_n)$ is constructed as the convolution of two cropped sequences $(\eta_r)$ 
with $|\eta_r| \leq \tau(r)^j$, and $(\gamma_m)$ with 
$|\gamma_m| \leq \tau(m)^j$. Specifically, with $R\geq 1$, $M\geq 1$, take 
\begin{equation}
 \beta_n
 = \sum_{\substack{rm = n \\ r \leq R, \, m \leq M}} \eta_r \gamma_m,
\label{e4.7.2}
\end{equation}
By Corollary~\ref{cor411estBetaSW} it follows that if $(\gamma_m)$ satisfies the S-W condition in the segment $1 \leq m \leq M$, then $(\beta_n)$ satisfies the S-W condition in the segment $1 \leq n \leq N = RM$, provided $\log M \gg \log 2N$. We have already used these observations along the lines of Case (C2) in the 
proof of Proposition\ref{p34LambdaSumRichEst}. 
\end{examps}

If $(\beta_n)$ satisfies the S-W condition in the segment $1 \leq n \leq N$, then it is natural to expect that the sum~\eqref{e41bilForm} is well-approximated by
\[
 \frac{1}{\varphi(q)}
 \mathop{\sum \sum}_{\substack{m \leq M, \, n \leq N \\ (mn,q) = 1}} \alpha_m \beta_n,
\]
so our goal is to estimate the difference. The challenge is to get meaningful results which are valid for moduli $q$ as large as possible. Detecting the congruence $m n \equiv a \pmod{q}$ by Dirichlet characters, and then applying the large sieve inequality produces a remarkable result (see Theorem 9.16 of~\cite{FI2}):

\begin{prop} \label{pEstBilFormMax}
Let $\cA = (\alpha_m)$, $\cB = (\beta_n)$ be sequences of complex numbers with $|\alpha_m| \leq \tau(m)^j$ and $|\beta_n| \leq \tau(n)^j$. Suppose $\cB$ satisfies the S-W condition in the segment $1 \leq n \leq N$. Then we have
\begin{equation}
 \sum_{q \leq Q} \max_{(a,q)=1} \Bigg|
  \mathop{\sum \sum}_{\substack{m \leq M, \, n \leq N \\ mn \equiv a \, (q)}} \alpha_m \beta_n
  - \frac{1}{\varphi(q)} \mathop{\sum \sum}_{\substack{m \leq M, \, n \leq N \\ (mn,q) = 1}} \alpha_m \beta_n
  \Bigg|
 \ll x (\log x)^{-A}
\label{e46bilFormCopQ}
\end{equation}
where $Q = x^{\frac{1}{2}} (\log x)^{-B}$, $x = MN \geq 2$, provided $M,N \geq x^\eps$ with $\eps > 0$. Here $\eps,A,j$ are any positive numbers, $B = B(A)$ depends on $A$, and the implied constant depends only on $\eps$, $A$, and $j$.
\end{prop}

\begin{rems}
Using the same tools, one can establish the more robust estimate
\begin{equation}
 \mathop{\sum \sum}_{\substack{q \leq Q, \, m \leq M \\ (q,m)=1}} \max_{(a,q)=1} \Bigg|
  \sum_{\substack{n \leq N \\ mn\equiv a \, (q)}} \beta_n
  - \frac{1}{\varphi(q)} \sum_{\substack{n \leq N \\ (n,q) =1}} \beta_n
  \Bigg|
 \ll x (\log x)^{-A},
\end{equation}
subject to the same conditions as for~\eqref{e46bilFormCopQ} with the extra 
condition $N \geq M$. \emph{Hint:} Relax the symbol $\max|*|$ by introducing 
factors $c_m(q) = \pm 1$ and follow the arguments in the proof of Theorem 
9.16 of~\cite{FI2}. Apply the mean value theorem for the character sums 
$\sum \chi(m) c_m(q)$ and the large sieve inequality for the character 
sums $\sum \chi(n) \beta_n$.
\smallskip

It is the use of Dirichlet characters in conjunction with the large sieve 
inequality that sets the limit $q \leq Q = x^\frac{1}{2} (\log x)^{-B}$ 
in~\eqref{e46bilFormCopQ}. In the works \cite{FoIw}, 
\cite{BFI1}, 
different methods were developed. 
These are based on the ideas of the dispersion method of Yu. V. Linnik~[L] 
and exponential sum bounds derived from the Riemann hypothesis for 
varieties and from the spectral theory of automorphic forms.  
This produced results for moduli $q$ significantly 
exceeding the barrier $x^{\frac{1}{2}}$,   
provided the residue class is fixed, that is,
 as $q$ varies $a$ does not move.
\smallskip

After the work of Goldston-Pintz-\Yil \cite{GPY}, it became desirable to have similar results with $a \pmod{q}$ varying in a specific fashion. Precisely, $a \pmod{q}$ is required to run over the roots of the polynomial
\begin{equation}
 Z(X)
 = \prod_{\substack{h' \in \cH \\ h' \neq h}} (X + h' - h), 
\label{e48Zdef}
\end{equation}
where $\cH$ is a fixed set of integers. Because $Z(X)$ is the product of linear polynomials, there would be no issue if the modulus $q = p$ was prime; simply $a$ runs over the finite set of numbers $h - h'$ which do not change with $p$. The problem becomes intricate for composite moduli $q$ since the residue class $a \pmod{q}$ can be very large. Yitang Zhang~\cite{Z} has shown how to deal with it by a clever application of the Chinese Remainder Theorem before opening the dispersion. His estimates hold only for special moduli (free of large prime divisors), but it is all that he needs to make a happy ending to this fascinating saga.

In the following sections we reproduce Zhang's estimates along his lines (which are reminiscent of the lines in our earlier works), but we proceed differently in details to enhance the presentation. Like Y. Zhang, we do not strive for the strongest quantitative results. In fact, we shall sacrifice strength to gain even a bit of clarity.
\end{rems}


\subsection{Bilinear forms over product moduli}

We are interested in the bilinear form \eqref{e41bilForm} to modulus~$q$ which is squarefree and admits a well-controlled factorization. We assume that
\begin{equation}
 q =
 rs \text{ squarefree},
\end{equation}
where $r,s$ run over the dyadic segments $R \leq r < 2R$, $S \leq s < 2S$ with $R \geq 1$, $S \geq 1$ to be restricted later. Therefore, $r,s$ are squarefree, co-prime and $q \asymp RS = Q$. By the Chinese Remainder Theorem one 
can cover uniquely all the reduced classes modulo $q$ by the 
pairs of reduced classes $a \pmod{r}$ and $b \pmod{s}$.

To every $s$ as above, 
we associate an arithmetic function $\gamma_s(b)$ such that
\begin{gather}
 \gamma_s(b) \text{ is periodic in } b \text{ of period } s,
\label{e410GammaPeriodic}
\\
 |\gamma_s(b)| \leq 1,
\qquad
 \gamma_s(b) = 0 \text{ if } (b,s) \neq 1,
\label{e411GammasDefSupp}
\\
 \sum_{b \,(\mo s)} |\gamma_s(b)|
 \leq \rho(s),
\label{e412GammasBd}
\end{gather}
where $\rho(s)$ is a multiplicative function bounded by $\tau(s)^j$ for some 
$j \geq 0$. Actually, we allow $\gamma_s(b)$ to depend on $r,a$. However, 
one should remember that $\gamma_s(b)$ are defined only if $rs$ is squarefree, 
a condition which will be often tacitly assumed in our exposition, 
but recalled only occasionally.

Given the squarefree number $r$ and $(a,r) = 1$, our goal is to evaluate the bilinear forms
\begin{equation}
 \mathop{\sum \sum}_{mn \equiv a \, (r)} \alpha_m \beta_n z(mn)
\label{e413bilGamma}
\end{equation}
for quite general coefficients $\alpha_m,\beta_n$ satisfying~\eqref{e42BilFormAlpBd}, \eqref{e43BilFormBetaBd}. Here the extra factor
\begin{equation}
 z(mn)
 = \sum_s \gamma_s(mn)
\label{e44defz}
\end{equation}
is not a new device; it appears also between the lines in~\cite{Z}, but not in such an abstract shape. Naturally one expects that~\eqref{e413bilGamma} is well-approximated by
\begin{equation}
 \frac{1}{\varphi(r)} \dsum_{(mn,r)=1} \alpha_m \beta_n z(mn),
\label{eBilFormZAppr}
\end{equation}
hence our goal is to estimate the difference. In fact, the coefficients $\alpha_m$ play only a structural function, but nothing particular. Therefore, we shall succeed by using only the upper bound~\eqref{e42BilFormAlpBd}. We are going to treat
\begin{equation}
 E(r,a)
 = \sum_{\substack{m \sim M \\ (m,r)=1}} \tau(m)^j \Bigg|
  \sum_{\substack{n \\ mn \equiv a \, (r)}} \beta_n z(mn)
  - \frac{1}{\varphi(r)} \sum_{\substack{n \\ (n,r)=1}} \beta_n z(mn)
  \Bigg|.
\label{e415defEra}
\end{equation}
Recall that the restriction $m \sim M$ means $M \leq m < 2M$. From now on we also assume that $\cB = (\beta_n)$ is supported on $N \leq n < 2N$.

We shall be obliged to perform an averaging of $E(r,a)$ over $r$ only in a few spots and to exploit the S-W condition for $\cB$ essentially to match the relevant (predicted) main terms. Here are the basic results:

\begin{prop} \label{pEstMaxE}
Let $R \geq 1$, $S \geq 1$, $M \geq 1$, $N \geq 2$. Suppose $|\beta_n| \leq \tau(n)^j$ and that $\cB = (\beta_n)$ satisfies the S-W condition in the interval $N \leq n < 2N$. Then we have
\begin{equation}
 \sum_{\substack{r \sim R\\ r \text{ squarefree}}} \tau(r)^j
  \max_{(a,r)=1} E(r,a)
 \ll x ( \log x )^{-A}
\label{eEstMaxEra}
\end{equation}
with $x = MN$, for any $A \geq 1$, subject to the following restrictions:
\begin{gather}
 (R+S) x^\eps < N < (RS)^{\frac{3}{2}},
\label{eEstECond1}
\\
 RN < x^{1-\eps}
\qquad \text{and} \qquad
 R^\frac{5}{4} S^\frac{11}{4} N^\frac{1}{2} < x^{1-\eps},
\\
 RSN < x^{1-\eps}
\qquad \text{or} \qquad
 R^\frac{5}{4} S^3 N^\frac{1}{2} < x^{1-\eps}.
\label{eEstECond3}
\end{gather}
Here $\eps$ is any positive number and the implied constant in~\eqref{eEstMaxEra} depends only on $\eps, A, j$.
\end{prop}

\begin{cor} \label{corEstBilForm}
Let $\cA = (\alpha_m)$, $\cB = (\beta_n)$ be sequences of real numbers with
\begin{gather*}
 |\alpha_m| \leq \tau(m)^j,
  \qquad
 M \leq m < 2M,
  \qquad
 M \geq 1,
\\
 |\beta_n| \leq \tau(n)^j,
  \qquad
 N \leq n < 2N,
  \qquad
 N \geq 2,
\end{gather*}
and $\alpha_m = 0$, $\beta_n = 0$ elsewhere. Let $x = MN$, $Q = RS$ with $R \geq 1$, $S \geq 1$, and $z \geq 2$. Suppose the following restrictions hold:
\begin{gather}
 N \ll x^\frac{1}{2},
  \qquad
 z^{-1} N \ll x^\eps R \ll N
\label{eEstBilFormCondNx}
\\
 z Q \ll x^{\frac{9}{17}-\eps},
  \qquad
 z^\frac{3}{2} Q^\frac{11}{4} \ll N x^{1-3\eps}
\label{eEstBilFormCondzQ}
\end{gather}
with some $\eps > 0$ and some implied constants which may depend on $\eps$. Suppose $\cB = (\beta_n)$ satisfies the S-W condition in $N \leq n < 2N$. Then we have
\begin{equation}
\sum_{r\sim R}  \max_{(a,r)=1}
\sum_{\substack{s \sim S \\ rs \text{ squarefree}}} 
\max_{(b,s)=1} \left|
   \dsum_{\substack{mn \equiv a \, (r)\\mn \equiv b \, (s)}}\alpha_m \beta_n
   - \frac{1}{\varphi(rs)} \dsum_{(mn,rs)=1} \alpha_m \beta_n
   \right|
 \ll x (\log x)^{-A}
\label{eEstBilForm}
\end{equation}
for any $A \geq 1$, the implied constant depending only on $\eps,A,j$.
\end{cor}

\begin{proof}[Proof of Corollary~\ref{corEstBilForm}]
First observe that we can assume 
\begin{equation}
 Q = RS > x^\frac{1}{2} (\log x)^{-B}
\label{eCorEstBilQRS}
\end{equation}
for some $B = B(A) > 0$ by applying Proposition~\ref{pEstBilFormMax} in the opposite case (this requires $M,N \geq x^\eps$, which hold due to~\eqref{eEstBilFormCondNx}).

The point $b$ at which the maximum is attained depends on $r$, $a$, and $s$. Given $r$ squarefree, $r \sim R$, $(a,r)=1$ and $s$ squarefree, $s \sim S$, $(r,s)=1$, we define the function $\gamma(b)$ in $b \pmod{s}$ by $\gamma(b) = \pm 1$ if $b$ 
 maximizes the absolute value among $(b,s)=1$ and we put $\gamma(b) = 0$ otherwise. Here $\pm 1$ is the sign of the innermost difference. Obviously $\gamma(b) = \gamma_s(b) = \gamma_s(b;r,a)$ satisfies the conditions~\eqref{e410GammaPeriodic}-\eqref{e412GammasBd} with $\rho(s) = 1$. Now we can write the above  
sum in the following form: 
\[
 \sum_{r \sim R} \max_{(a,r)=1} \sum_{s \sim S} \left(
  \dsum_{mn \equiv a \, (r)} \alpha_m \beta_n \gamma_s(mn)
  - \frac{1}{\varphi(rs)} \dsum_{(mn,rs)=1} \alpha_m \beta_n
 \sumst_{b \, (s)} \gamma_s(b)  \right).
\]
Changing the order of summation we estimate the sum~\eqref{eEstBilForm} by
\begin{equation}
 \sum_{r \sim R} \max_{(a,r)=1} \sum_{\substack{m \sim M \\ (m,r)=1}} \tau(m)^j \left|
  \sum_{mn \equiv a \, (r)} \beta_n z(mn)
  - \frac{1}{\varphi(r)} \sum_{(n,r)=1} \beta_n \sum_{(s,mn)=1} \frac{1}{\varphi(s)}
   \sumst_{b \, (s)} \gamma_s(b)
  \right|.
\label{eEstBilFormStep}
\end{equation}
Recall our convention that $\gamma_s(b)$ is defined only if $rs$ is 
squarefree, so~\eqref{eEstBilFormStep} has the hidden condition $(r,s)=1$.

We are approaching the sum~\eqref{e415defEra}. The main term in~\eqref{eEstBilFormStep} is not exactly the one we have assumed in~\eqref{e415defEra}, but it is very close. Indeed, if we write
\[
 \sumst_{b \, (s)} \gamma_s(b)
 = \sumst_{b \, (s)} \gamma_s(mb)
\]
in~\eqref{eEstBilFormStep} and go back to~\eqref{e44defz} for $z(mn)$ in the main term of~\eqref{e415defEra}, then it is clear that the difference between the main terms is
\[
 \frac{1}{\varphi(r)} \sum_s 
 \sumst_{b \, (s)} \gamma_s(mb) \Bigg(
  \sum_{\substack{(n,r)=1 \\ n \equiv b \, (s)}} \beta_n
  - \frac{1}{\varphi(s)} \sum_{(n,rs)=1} \beta_n
  \Bigg).
\]
This difference contributes to~\eqref{eEstBilFormStep} at most
\begin{multline*}
 \Delta(R,S,M,N)
 =
\\
 \sum_r \frac{1}{\varphi(r)} \max_{(a,r)=1}
  \sum_s \sumst_{b \, (s)} \left( \sum_m \tau(m)^j |\gamma_s(mb)| \right) \Bigg|
    \sum_{\substack{(n,r)=1 \\ n \equiv b \,(s)}} \beta_n
    - \frac{1}{\varphi(s)} \sum_{(n,rs)=1} \beta_n
   \Bigg|.
\end{multline*}
Note that
\begin{align*}
 \sum_s \sumst_{b \, (s)} \left( \sum_m \tau(m)^j |\gamma_s(mb)| \right)^2
 &\leq \sum_s \sum_{m' \equiv m \, (s)} \big( \tau(m') \tau(m) \big)^j
\\
 &\ll M^2 (\log 2M)^{4^j} ,
\end{align*}
because $s \sim S$ and $S \ll M^\frac{1}{2}$ by the conditions~\eqref{eEstBilFormCondNx}, \eqref{eEstBilFormCondzQ}. Hence, by Cauchy's inequality
\[
 \Delta(R,S,M,N)^2
 \ll M^2 (\log 2M)^{4^j} \sum_r \frac{1}{\varphi(r)} \Delta_r(S,N),
\]
where
\[
 \Delta_r(S,N)
= \sum_{s \sim S} \sumst_{b \, (s)} \Bigg|
  \sum_{\substack{(n,r)=1 \\ n \equiv b \, (s)}} \beta_n
  - \frac{1}{\varphi(s)} \sum_{(n,rs)=1} \beta_n
  \Bigg|^2
\ll \tau(r)^2 N^2 (\log N)^{-A}
\]
by Corollary~\ref{cor42BdBetaSW} (note that $S \ll N^\frac{1}{2}$ by the conditions~\eqref{eEstBilFormCondNx}, \eqref{eEstBilFormCondzQ}). This shows that
\begin{equation}
 \Delta(R,S,M,N)
 \ll MN (\log N)^{-A},
\end{equation}
which is admissible for~\eqref{eEstBilForm}.

To complete the proof of Corollary~\ref{corEstBilForm} it remains to check 
that the conditions~\eqref{eEstBilFormCondNx},~\eqref{eEstBilFormCondzQ} 
imply the conditions~\eqref{eEstECond1}-\eqref{eEstECond3} 
so that~\eqref{eEstMaxEra} can be applied (the $\eps$ can be different in 
the two sets of conditions). The task is not hard, so we leave it for 
the reader (see~\eqref{eCorEstBilQRS}).
\end{proof}

\section{\bf{The Dispersion Method}}
\label{c5disp}

\subsection{Statement of results}

We shall prove Proposition~\ref{pEstMaxE} by the dispersion method. First, 
applying Cauchy's inequality to~\eqref{e415defEra} we get
\begin{equation}
 E^2(r,a)
 \ll \cD(r,a) M (\log M)^{4^j},
\end{equation}
where
\begin{equation}
 \cD(r,a)
 = \sum_{(m,r)=1} f \left( \frac{m}{M} \right) \left|
  \sum_{mn \equiv a \,(r)} \beta_n z(mn)
  - \frac{1}{\varphi(r)} \sum_{(n,r)=1} \beta_n z(mn)
  \right|^2
\label{eDispDef}
\end{equation}
and $f$ is a smooth function compactly supported on~$\R^+$.

We shall be able to handle the dispersion~$\cD(r,a)$ for any single class 
$a \pmod{r}$, except for matching the main terms. Therefore, we state the 
results in two stages.

\begin{prop} \label{p43DispBdE}
Let $r$ be squarefree, $(a,r)=1$. Suppose $S \geq 1$, $M \geq 1$, $N \geq 2$, 
and
\begin{equation}
 (r+S) x^\eps
 \ll N
 \ll (rS)^{\frac{3}{2}},
  \qquad
 \text{with } x = MN.
\label{e5.3}
\end{equation}
Then, we have
\begin{multline}
 \cD(r,a)
 \ll \frac Mr \sum_{k<2S}\frac{\tau(k)^j}{k}\bigg(\cE_k(r)+\frac{1}{\varphi(r)}
\cE_{kr}(1)\bigg) (\log N)^{2^j}  
\\
  + \{ r^{-\frac{3}{2}} M^\frac{1}{2} N^2 
+ r^{-\frac{3}{4}} S^\frac{11}{4} N^\frac{3}{2} \} x^\eps 
  + \min \{r^{-\frac{3}{2}} S^\frac{1}{2} M^\frac{1}{2} N^2, r^{-\frac{3}{4}} S^3 N^\frac{3}{2} \} x^\eps , 
\label{eDispBd}
\end{multline}
where 
\begin{equation}
 \cE_k(r) = \sum_{c<2S}\frac{\tau(c)^j}{c}\sumst_{\nu (cr)}
\bigg(\sum_{\substack{(n,k)=1\\ n\equiv \nu (cr)}} \beta_n 
- \frac{1}{\varphi (cr)}\sum_{(n,ckr)=1}\beta_n\bigg)^2
\end{equation}
with implied constants depending on $\eps$ and $j$. Moreover, 
if the sequence $(\beta_n)$ satisfies the S-W condition in the segmant 
$N\le n <2N$, then
\begin{equation}
 \sum_{r < R} \cE_k(r)
 \ll \tau(k)N^2(\log x)^{-A}
\label{eDispErrBd}
\end{equation}
for $R\ll Nx^{-\varepsilon}$, with any $A \geq 1$, the implied constant 
depending on $A$, $\eps$, and $j$.
\end{prop}

Now we can derive Proposition~\ref{pEstMaxE} from Proposition~\ref{p43DispBdE}. By Cauchy's inequality, the square of the sum~\eqref{eEstMaxEra} is bounded by
\[
 RM (\log R)^{4^j} (\log M)^{4^j} \sum_{r \sim R} \max_{(a,r)=1} \cD(r,a).
\]
Applying~\eqref{eDispBd} and~\eqref{eDispErrBd} to the above, we find
\begin{multline*}
 RMR^{-1} N x (\log x)^{-A}
 + RM \{ R^\frac{1}{2} M^\frac{1}{2} N^2 + R^\frac{1}{4} S^\frac{11}{4} N^\frac{3}{2} \} x^\eps
\\
 + RM \min \{ R^{-\frac{1}{2}} S^\frac{1}{2} M^\frac{1}{2} N^2, \, R^\frac{1}{4} S^3 N^\frac{3}{2} \} x^\eps
 \ll x^2 (\log x)^{-2A}
\end{multline*}
subject to the conditions~\eqref{eEstECond1}-\eqref{eEstECond3}. This proves~\eqref{eEstMaxEra}.

Now we proceed to the proof of Proposition~\ref{p43DispBdE}, which is 
the gist of this project. We open the square in~\eqref{eDispDef} 
and arrange the dispersion into three parts:
\begin{equation}
 \cD(r,a)
 = \cD_1 - 2\cD_2 + \cD_3
\label{eDispSplitD123}
\end{equation}
with the purpose of evaluating each part asymptotically so that the resulting main terms match and cancel out.

\subsection{Evaluation of $\cD_1$}
\label{D1}

We start from $\cD_1$, because it is the hardest part and because the other two parts $\cD_2,\cD_3$ will make use of the arguments which appear in our treatment of $\cD_1$. We have
\begin{equation}
\begin{aligned}
 \cD_1
 &= \sum_{(m,r)=1} f \left(\frac{m}{M}\right) \left(
  \sum_{mn \equiv a \, (r)} \beta_n \left( \sum_s \gamma_s(mn) \right)
  \right)^2
\\
 &= \mathop{\sum \sum}_{n \equiv n' \, (r)} \beta_n \beta_{n'}
   \sum_s \sum_{s'} \sum_{mn \equiv a \,(r)}
    f\left(\frac{m}{M}\right) \gamma_s(mn) \gamma_{s'}(mn').
\end{aligned}
\end{equation}
Note that the coprimality conditions $(m,rss') = (n,rs) = (n',rs') = 1$ and $(ss',r)=1$ are redundant with our conditions imposed on $\gamma_s(n)$. We split the summation over $m$ into residue classes
\[
 \mu
 \equiv a \cl{n [s,s']} [s,s']
  + \omega \cl{r} r
  \pmod{r [s,s']},
\]
where $\omega$ runs over the reduced classes (mod $[s,s']$), and for each class we apply the Poisson formula
\[
 \sum_{m \equiv \mu \, (r [s,s'])} f(m/M)
 = \frac{M}{r [s,s']}
  \sum_h \hat{f} \left( \frac{hM}{r [s,s']} \right) \e \left(\frac{-\mu h}{r [s,s']}\right)
\]
giving
\begin{multline}
 \cD_1
 = \frac{M}{r} \dsum_{n \equiv n' \, (r)} \beta_n \beta_{n'} \sum_s \sum_{s'} \sum_{\omega \, ([s,s'])}
  \gamma_s(\omega n) \gamma_{s'} (\omega n') \big/ [s,s']
\\
 \cdot \sum_h \hat{f} \left( \frac{hM}{r [s,s']} \right) \e \left(-ah \frac{\cl{n[s,s']}}{r} - \frac{\omega h \cl{r}}{[s,s']}\right)
\label{e436D1byPoisson}
\end{multline}
The main term will emerge from the zero frequency, and we denote its contribution by
\begin{equation}
 \cD_{10}
 = \hat{f}(0) \frac{M}{r}
  \dsum_{n \equiv n' \, (r)} \beta_n \beta_{n'}
  \sum_s \sum_{s'} \frac{1}{[s,s']}
  \sum_{\omega \, (\mo [s,s'])} \gamma_s(\omega n) \gamma_{s'}(\omega n').
\end{equation}
We denote by $\cD_{11}$ the contribution of the terms with $h \neq 0$ to~\eqref{e436D1byPoisson}, so we have
\begin{equation}
 \cD_1
 = \cD_{10} + \cD_{11}.
\end{equation}

The complete sum over $\omega \pmod{[s,s']}$ in $\cD_{10}$ depends only on the residue class of $n' \cl{n}$ modulo
\begin{equation}
 c
 = (s,s').
\label{edefc}
\end{equation}
Precisely, if $(n,s)= (n',s') = 1$ and $n' \equiv e n \pmod{c}$, then
\begin{equation}
 \sum_{\omega \, (\mo [s,s'])}
  \gamma_s(\omega n) \gamma_{s'}(\omega n')
 = \sum_{\omega \, (\mo [s,s'])}
  \gamma_s(\omega) \gamma_{s'} (\omega e) = [s,s']\sigma_{s,s'}(e), 
\label{e440gammasbyn'e}
\end{equation} 
the last equation being the definition of $\sigma_{s,s'}(e)$. 

\begin{proof}[Proof of \eqref{e440gammasbyn'e}]
This is an easy exercise in using the periodicity of $\gamma_s(b)$. Write $\omega = \alpha s' c^{-1} + \beta s$ with $\alpha \pmod{s}$ and $\beta \pmod{s' c^{-1}}$ being primitive classes. Then, we have
\begin{align*}
 \gamma_s(\omega n) \gamma_{s'} (\omega n')
 &= \gamma_s(\alpha s' e^{-1} n) \gamma_{s'} \big( (\alpha s' c^{-1} + \beta s ) n' \big)
\\
 &= \gamma_s(\alpha s' c^{-1} n) \gamma_{s'} (\alpha s' c^{-1} e n + \beta s n').
\end{align*}
Here, summing over $\beta$ we can replace $n'$ by $en$. This proves~\eqref{e440gammasbyn'e}.
\end{proof}

Note that
\begin{equation}
 [s,s'] \sum_{e \, (c)} \sigma_{s,s'}(e)
 = \gamma(s) \gamma(s') ,
\label{e442SigmaSum}
\end{equation}
where $c$ and $\sigma_{s,s'}(e)$ are defined 
by~\eqref{edefc},~\eqref{e440gammasbyn'e} respectively, and 
\begin{equation}
 \gamma(s)
 = \sum_{\omega\,(s)} \gamma_s(\omega).
\end{equation}

By~\eqref{e440gammasbyn'e} we get
\[
 \cD_{10}
 = \hat{f}(0) \frac{M}{r} \sum_s \sum_{s'} \sum_{e \, (c)}
  \sigma_{s,s'} (e) \dsum_{\substack{n' \equiv n \, (r) \\ n' \equiv en \, (c)}} \beta_n \beta_{n'}.
\]
Do not forget the hidden coprimality conditions $(n,rs) = (n',rs') = 1$. The sum of $\beta_n \beta_{n'}$ can be arranged as follows:
\begin{multline*}
 \dsum_{\substack{n' \equiv n \, (r) \\ n' \equiv en \, (c)}} \beta_n \beta_{n'}
 = \frac{1}{\varphi(rc)} \left( \sum_{(n,rs)=1} \beta_n \right) \left( \sum_{(n',rs')=1} \beta_{n'} \right)
\\
  + \sumst_{u \, (r)} \sumst_{v \, (c)} \Bigg(
   \sum_{\substack{n \equiv u \, (r) \\ n \equiv v \, (c)}} \beta_n
   - \frac{1}{\varphi(rc)} \sum_{(n,rs)=1} \beta_n
   \Bigg)
\Bigg(
   \sum_{\substack{n' \equiv u \, (r) \\ n' \equiv ev \, (c)}} \beta_{n'}
   - \frac{1}{\varphi(rc)} \sum_{(n',rs') = 1} \beta_{n'}
   \Bigg).
\end{multline*} 
Hence, we obtain
\begin{equation}
 \cD_{10}
 = \cD^*_{10} + \cD^+_{10},
\end{equation}
where $\cD^*_{10}$ is the main term
\begin{equation}
 \cD^*_{10}
 = \hat{f}(0) \frac{M}{r \varphi(r)} \sum_s \sum_{s'} \frac{\gamma(s) \gamma(s')}{\varphi(c) [s,s']}
  \left( \sum_{(n,rs)=1} \beta_n \right)
  \left( \sum_{(n',rs')=1} \beta_{n'} \right),
\label{e445D10*def}
\end{equation}
and the remaining part is given by
\begin{multline}
 \cD^+_{10}
 = \hat{f}(0) \frac{M}{r} \sumst_{u \, (r)} \sum_s \sum_{s'} 
\sumst_{v \, (c)} \sumst_{v' \, (c)}
  \sigma_{s,s'}(v/v')
\\
  \cdot \Bigg( \sum_{\substack{n \equiv u \, (r) \\ n \equiv v \, (c) \\ (n,rs)=1}} \beta_n
   - \frac{1}{\varphi(rc)} \sum_{(n,rs) = 1} \beta_n
   \Bigg)
  \Bigg( \sum_{\substack{n \equiv u \, (r) \\ n \equiv v' \, (c) \\ (n,rs')=1}} \beta_n
   - \frac{1}{\varphi(rc)} \sum_{(n,rs') = 1} \beta_n
   \Bigg). 
\label{e5.18}
\end{multline}
The main term $\cD^*_{10}$ will be matched with the main terms of $\cD_2,\cD_3$ 
but the remaining part $\cD^+_{10}$ can only be estimated completely after 
summation over $r \sim R$, because we have to apply the 
property~\eqref{e44SWcond} for moduli $cr$. 

By the inequality $2|AB| \leq A^2 + B^2$ we find that the quadruple sum 
over $s$, $s'$, $v \, (c)$, $v' \, (c)$ in~\eqref{e5.18} is bounded by
\[
 \sum_s \sum_{c \mid s} \sumst_{v \, (c)} \sum_{(s',s) = c}\, \sum_{v' \, (c)}
 |\sigma_{s,s'}(v'/v)| \Bigg(
  \sum_{\substack{n \equiv u \, (r) \\ n \equiv v \, (c) \\ (n,rs) = 1}} \beta_n
  - \frac{1}{\varphi(rc)} \sum_{(n,rs)=1} \beta_n
  \Bigg)^2.
\]
Here the sum over $v' \pmod{c}$ is bounded by (see~\eqref{e442SigmaSum}, 
\eqref{e412GammasBd})
\begin{align*}
 \sumst_{e \, (c)} |\sigma_{s,s'}(e)|
 &\leq \frac{1}{[s,s']} \sum_{\omega \, ([s,s'])} \sumst_{e \, (c)}
  \big| \gamma_s(\omega) \gamma_{s'} (\omega e) \big|
\\
 &= \frac{1}{[s,s']}
  \left( \sum_{\omega\,(s)} |\gamma_s(\omega)| \right)
  \left( \sum_{\omega \, (s')} |\gamma_{s'} (\omega)| \right)
  \leq \frac{\rho(s) \rho(s')}{[s,s']}.
\end{align*}
Now put $s = ck$, $s' = cl$. Summing over $k,l < 2S$ we conclude the following clear estimate for $\cD^+_{10}$;
\[
 |\cD^+_{10}|
 \leq |\hat{f}(0)| M \mathop{\sum \sum \sum}_{c,k,l < 2S} \frac{\rho^2(c) \rho(kl)}{crkl}
  \cdot \sumst_{\nu \, (cr)} \Bigg(
   \sum_{\substack{(n,crk)=1 \\ n \equiv \nu \, (cr)}} \beta_n
   - \frac{1}{\varphi(cr)} \sum_{(n,crk)=1} \beta_n
   \Bigg)^2.
\]
Hence, 
\begin{equation}
\cD^+_{10} \ll \frac{M}{r} \Bigg(\sum_{l<2S}\frac{\tau(l)^j}{l}\Bigg)
\sum_{k<2S}\frac{\tau(k)^j}{k}\cE_k(r) , 
\end{equation} 
which yields the first part on the right side of~\eqref{eDispBd}.

Using Corollary~\ref{cor42BdBetaSW}, we can estimate $\cE(r)$ on average 
over $r < R$ as follows. 
\begin{align*}
 \sum_{r < R}\cE_k(r)  
& \le \frac{2\log 2S}{C}
\sum_{q \leq CR} \tau(q)^{j+1} \sumst_{\nu \, (q)}  \Bigg( 
  \sum_{\substack{n \leq N \\ n \equiv \nu\,(q) \\ (n,k) = 1}} \beta_n
  - \frac{1}{\varphi(q)} \sum_{\substack{n \leq N \\ (n,qk) = 1}} \beta_n
  \Bigg)^2 \\
&  \ll \frac{\tau(k)}{C} \big( CR + \sqrt{CRN} + N (\log N)^{-A} \big) 
N (\log 2QN)^{4^{j+2} } ,
\end{align*}
where $1 \leq C \leq 2S$. Since $j$ is fixed and $A$ is arbitrary, 
we obtain~\eqref{eDispErrBd}.



\subsection{First estimation of $\cD_{11}$}
\label{D11}

Recall that $\cD_{11}$ is given by~\eqref{e436D1byPoisson} with $h \neq 0$. We 
shall apply two methods for the estimation of $\cD_{11}$. We begin with the one 
which is more complex. The second method will be easier because it makes use 
of the basic transformations in the first one.

We split the summation over $n',n$ according to the difference, say
\begin{equation}
 n'
 = n+u.
\label{e450nn'}
\end{equation}
Therefore, for every given $u$, the variable $n'$ will be uniquely determined
 by $n$. Let $\cU$ be the set of shifts $u$, so $\cU \subseteq (-N, N)$, and
 in case of $\cD_{11}$ we have $u \equiv 0 \pmod{r}$. However, this property
 of $u$ does not play any role, except that it ensures
 $|\cU| \leq 2 N r^{-1} + 1 \leq 3 N r^{-1}$. Because the cardinality of 
$\cU$ is the only factor in the result, we shall be able to treat similarly 
the corresponding sum $\cD_{21}$ in subsection~\ref{s48estD21}. Now, 
by~\eqref{e436D1byPoisson} we obtain
\begin{equation}
\begin{aligned}
 |\cD_{11}|
 &\leq \frac{2M}{r} \sum_{u \in \cU} \dsum_{(n,r) = 1} |\beta_n \beta_{n'}|
    \Bigg| \sum_s \sum_{s'} \sum_\omega \sum_{h \neq 0} \Bigg| 
\\
 &\ll r^{-2} M N^{\frac{3}{2}} (\log N)^j T(u)^{\frac{1}{2}}
\end{aligned}
\label{e451D11est}
\end{equation}
for some $u \in \cU$ with
\[
 T(u)
 = \sum_{(n,r)=1} f \left( \frac{n}{N} \right) \Bigg|
  \sum_s \sum_{s'} \sum_\omega \sum_{h \neq 0} 
  \Bigg|^2.
\]
Here, the coefficients $\beta_n, \beta_{n'}$ are gone; they are replaced with a nice smooth function $f$ supported in $[\frac{1}{2}, 3]$. Squaring out, we get
\begin{multline}
 T(u)
 = \sum_{(n,r)=1} f\left(\frac{n}{N}\right)
  \sum_s \sum_{s'} \sum_{s_1} \sum_{s_1'} \frac{1}{[s,s'][s_1,s_1']}
\\
  \sum_{\omega \, ([s,s'])} \gamma_s(\omega n) \gamma_s(\omega n')
  \sum_{\omega_1 \, ([s_1,s_1'])} \gamma_{s_1} (\omega_1 n) \gamma_{s_1'} (\omega_1 n')
\\
 \dsum_{h h_1 \neq 0}
  \hat{f} \left( \frac{hM}{r [s,s']} \right)
  \bar{\hat{f}} \left( \frac{h_1 M}{r [s_1, s_1']} \right)
  \e \left(
   \frac{-a \cl{n}}{r} \big( h \conj{[s,s']} - h_1 \conj{[s_1,s_1']} \big)
   - \frac{\omega h \conj{r}}{[s,s']}
   + \frac{\omega_1 h_1 \conj{r}}{[s_1,s_1']}
   \right).
\label{e452TuSumT=Tn}
\end{multline}
Put
\begin{equation}
 d
 = h [s_1, s_1'] 
  - h_1 [s,s'].
\label{e453DDef}
\end{equation}
Let $T^= (u)$ be the partial sum of $T(u)$ with $d = 0$, and $\Tn (u)$ the remaining sum.

We estimate $T^=(u)$ by summing all terms with absolute value. Given $n$ with $(n,ss_1) = 1$ and $(n', s's_1') = 1$ (recall~\eqref{e450nn'}) we find that the sum over $\omega \pmod{[s,s']}$ is bounded by $\rho(s) \rho(s')$ and the sum over $\omega_1 \pmod{[s_1,s_1']}$ is bounded by $\rho(s_1)\rho(s_1')$. Hence
\[
 T^=(u)
 \ll N \sum_s \sum_{s'} \sum_{s_1} \sum_{s_1'}
  \frac{\rho(s) \rho(s') \rho(s_1) \rho(s_1')}{[s,s'] [s_1,s_1']}
  \sumst_{h \neq 0} \left| \hat{f} \left( \frac{hM}{r[s,s']} \right) \right|,
\]
where $\sumst$ is restricted by the requirement that 
$h_1 = h [s_1,s_1'] / [s,s']$ is an integer. Hence, 
$[s_1,s_1'] = k [s,s'] / (h,[s,s'])$
where $k$ is an integer. We get $1 \leq k \leq 4 S^2$ and
\[
 T^=(u)
 \ll S^\eps N \sum_s \sum_{s'} \sum_{h \neq 0} \sum_{1 \leq k \leq 4 S^2}
  \frac{(h,[s,s'])}{k [s,s']^2} \left| \hat{f} 
\left( \frac{hM}{r [s,s']} 
\right) \right| \tau_3(k)
\]
because, for given $s,s',h,k$, the number of 
$s_1,s_1'$ is $\leq \tau_3(kss')\ll S^{\varepsilon}$. Hence
\[
 T^=(u)
 \ll S^\eps N \sum_{l < 4S^2} \sum_{h \neq 0}
  \frac{(h,l)}{l^2} \left| \hat{f} \left( \frac{hM}{rl}\right) \right| \big( \log 2 |h| \big)^4.
\]
This easily yields
\begin{equation}
 T^=(u)
 \ll r N M^{-1} (rS)^\eps.
\label{e454T=Bd}
\end{equation}

Now we proceed to $\Tn(u)$ which is given by~\eqref{e452TuSumT=Tn} 
with the summation condition $d = h [s_1,s_1'] - h_1 [s,s'] \neq 0$. We put
\begin{equation}
 \omega n
 \equiv \alpha \, (s),
\qquad
 \omega n'
 \equiv \alpha' \, (s'),
\qquad
 \omega_1 n
 \equiv \alpha_1 \, (s_1),
\qquad
 \omega_1 n'
 \equiv \alpha_1' \, (s_1'),
\label{e455ss's1s1'congs}
\end{equation}
and we extract from~\eqref{e452TuSumT=Tn}  the following formula for $ \Tn(u)$: 
\begin{multline}
\mathop{\sum \sum \sum \sum}_{(s s' s_1 s_1', r ) =1} \Bigg(
  \dsum_{\substack{\alpha \, (s) \\ \alpha_1 \, (s_1)}} \gamma_s(\alpha) \gamma_{s'}(\alpha') \big/ [s,s']
  \Bigg)
  \Bigg(
  \dsum_{\substack{\alpha_1 \, (s_1) \\ \alpha_1' \, (s_1')}} \gamma_{s_1}(\alpha_1) \gamma_{s_1'}(\alpha_1') \big/ [s_1,s_1']
  \Bigg)
\\
  \dsum_{\substack{hh_1 \neq 0 \\ d \neq 0}} \hat{f} \left( \frac{hM}{r[s,s']} \right) \hat{f} \left( \frac{h_1 M}{r [s_1,s_1']} \right)
\\
  \sum_{\substack{(n,rss_1)=1 \\ (n',rs's_1')=1}} f \left( \frac{n}{N} \right)
   \e \left(
    \frac{-ad}{r} \conj{n [s,s'] [s_1,s_1']}
    - \frac{\omega h \conj{r}}{[s,s']}
    + \frac{\omega_1 h_1 \conj{r}}{[s_1,s_1']}
    \right),
\label{e456TnSums}
\end{multline}
where $\omega \pmod{[s,s']}$ and $\omega_1 \pmod{[s_1,s_1']}$ are determined in terms of $n$ and $n' = n+u$ by the congruences~\eqref{e455ss's1s1'congs}.

The innermost sum over $n$ in~\eqref{e456TnSums} is an incomplete Kloosterman sum of modulus $rw$ with $w = [s,s',s_1,s_1']$. We shall see this clearly after completing the sum and factoring into prime moduli. To complete the sum, we split $n$ into residue classes modulo $rw$ and apply Poisson's formula. We find
\begin{equation}
 \sum_n f \left(\frac{n}{N}\right) \e \big( \,\,\, \big)
 = \frac{N}{rw} \sum_l \hat{f} \left( \frac{lN}{rw} \right) V_l(rw)
\label{e457PoissFE}
\end{equation}
where $V_l(rw)$ is the complete sum
\[
 V_l(rw)
 = \sum_{\nu \, (\mo \, rw)} \e \left(
  \frac{- \nu l}{rw}
  - \frac{ad}{r} \conj{\nu [s,s'] [s_1,s_1']}
  - \frac{\omega h \conj{r}}{[s,s']}
  - \frac{\omega_1 h_1 \conj{r}}{[s_1,s_1']}
  \right). 
\]
Here $\omega \pmod{[s,s']}$ and $\omega_1 \pmod{[s_1,s_1']}$ are determined in terms of $\nu$ and $\nu' = \nu + u$ by the congruences~\eqref{e455ss's1s1'congs} with $n,n'$ substituted by $\nu,\nu'$, respectively. The complete sum factors into
\[
 V_l(rw)
 = \cJ_l(w) K \big(l \conj{w}, ad \conj{[s,s'][s_1,s_1']}; r)
\]
where $K(*,*;r)$ is the Kloosterman sum to modulus $r$, and
\[
 \cJ_l(w)
 = \sum_{\substack{\nu \, (\mo \, w) \\ (\nu,ss_1 ) =1 \\ (\nu', s' s_1') = 1}}
  \e \left(
   \frac{- \nu l \conj{r}}{[s,s',s_1,s_1']}
   - \frac{\omega h \conj{r}}{[s,s']}
   + \frac{\omega_1 h_1 \conj{r}}{[s_1,s_1']}
   \right).
\]
By Weil's estimate we get
\[
 |K(l\conj{w}, *; r)|
 \leq (l,r)^{\frac{1}{2}} r^{\frac{1}{2}} \tau(r).
\]
For $l = 0$ we can do better by using estimates for Ramanujan's sums
\[
 |K(0,d;r)|
 \leq (d,r).
\]
Next $\cJ_l(w)$ factors into sums to distinct prime moduli $p \mid w$. If $p \mmid s s's_1 s_1'$, then the local sums appear in the form
\[
 \sum_{\nu \, (\mo \, p)} \e_p \big( -\nu l \conj{r w / p} + A \conj{\nu} + A' \conj{\nu'} \big)
\]
with $A A'=0$ and $\nu' = \nu + u$. These local sums are Kloosterman sums 
(if $A = 0$, then shift $\nu$ to $\nu - u$) which are bounded by 
$2 (l,p)^\frac{1}{2} p^\frac{1}{2}$. If $p^2 \mid s s' s_1 s_1'$, then 
we use the trivial bound $p$. Hence, we get
\[
 |\cJ_l(w)|
 \leq \tau(w) \prod_{p \mmid s s' s_1 s_1'} (l,p)^\frac{1}{2} p^\frac{1}{2}
  \prod_{p^2 \mid s s' s_1 s_1'} p
 \leq \tau(w) (l,w)^\frac{1}{2} (ss' s_1 s_1')^\frac{1}{2}.
\]
For $l = 0$ we do better; $|\cJ_0(w)| \leq \tau(w) w$. Multiplying the 
above local estimates we obtain
\begin{equation}
 |V_l(rw)|
 \leq \tau(rw) (l,rw)^\frac{1}{2} r^\frac{1}{2} (2S)^2
\label{e458VlBd}
\end{equation}
for every $l$, and for $l = 0$ we have
\begin{equation}
 |V_0(rw)|
 \leq \tau(w) w (d,r).
\label{e459V0Bd}
\end{equation}
Introducing~\eqref{e458VlBd} and~\eqref{e459V0Bd} into~\eqref{e457PoissFE} we get
\begin{equation}
 \sum_n f\left(\frac{n}{N}\right)
  \e \big( \, \, \, \big)
 \ll N \frac{(d,r)}{r} \tau(w)
  + S^2 r^\frac{1}{2} \tau_3(rw)
\label{e460SumFEestVl}
\end{equation}
where $d$ is defined in~\eqref{e453DDef}. Introducing~\eqref{e460SumFEestVl} into~\eqref{e456TnSums} we get
\begin{multline}
 \Tn(u)
 \ll \sum_s \sum_{s'} \sum_{s_1} \sum_{s_1'}
  \frac{\rho(s) \rho(s')}{[s,s']} \frac{\rho(s_1)\rho(s_1')}{[s_1,s_1']}
\\
  \dsum_{\substack{hh_1 \neq 0 \\ d \neq 0}} \left|
   \hat{f} \left( \frac{hM}{r[s,s']} \right) \hat{f} \left( \frac{h_1 M}{r [s_1,s_1']} \right)
   \right|
   \left( N \frac{(d,r)}{r} + S^2 r^\frac{1}{2} \right) \tau_3(rw).
\end{multline}
It is easy to see that the second term $S^2 r^\frac{1}{2}$ contributes at 
most (ignoring $d \neq 0$)
\begin{equation}
 S^2 r^\frac{1}{2} \left(\frac{r}{M}\right)^2
  \sum_s \sum_{s'} \sum_{s_1} \sum_{s_1'} \rho(s) \rho(s') \rho(s_1) \rho(s_1') \tau_3(rw)
 \ll r^\frac{5}{2} S^6 M^{-2} (rS)^\eps.
\label{eignoring}
\end{equation}
The first term $N(d,r) / r$ contributes at most (keeping $d \neq 0$)
\[
 (rS)^\eps \frac{N}{r}
 \dsum_{S \leq a,a_1 \leq 4S^2} (aa_1)^{-1}
 \dsum_{\substack{hh_1 \neq 0 \\ h a_1 \neq h_1 a}}
  \left(1 + \frac{|h| M}{ra} \right)^{-4} \left(1 + \frac{|h_1| M}{ra_1} \right)^{-4}
  (h a_1 - h_1 a, r).
\]
One can show by crude arguments that this is $\ll (rS)^\eps r S^4 N M^{-2}$, 
so it is smaller than~\eqref{eignoring} because of~\eqref{e5.3}. 
Adding the above bound~\eqref{eignoring} for $\Tn(u)$ to the 
bound~\eqref{e454T=Bd} for $T^=(u)$ we find that
\begin{equation}
 T(u)
 \ll (r N M^{-1} + r^\frac{5}{2} S^6 M^{-2}) (rS)^\eps.
\end{equation}
Hence, by~\eqref{e451D11est} we obtain
\begin{equation}
 \cD_{11}
 \ll \big(r^{-\frac{3}{2}} N^2 M^\frac{1}{2}
  + r^{-\frac{3}{4}} S^3 N^\frac{3}{2} \big)
  (rS)^\eps.
\label{e464D11Bd1}
\end{equation}


\subsection{Second estimation of $\cD_{11}$}
\label{D11'}

Recall again that $\cD_{11}$ is given by~\eqref{e436D1byPoisson} with $h \neq 0$. Now we put the variable $s'$ outside and apply Cauchy's inequality getting (in place of~\eqref{e451D11est})
\begin{equation}
 \cD_{11}
 \ll r^{-2} S M N^\frac{3}{2} (\log N)^j T(u;s')^\frac{1}{2}
\label{e465D11BdCauchy}
\end{equation}
where $T(u;s')$ is given by the same Fourier series as $T(u)$ in~\eqref{e452TuSumT=Tn}, but with $s' = s_1'$ being fixed, so no longer summation variables. As before, we split
\begin{equation}
 T(u;s')
 = T^=(u;s') + \Tn(u;s')
\end{equation}
according to $d = h [s_1, s_1'] - h_1 [s,s'] = h [s_1,s'] - h_1 [s,s']$ vanishing or not vanishing.

By the same arguments applied to $T^=(u)$ we derive
\begin{equation}
 T^=(u;s')
 \ll r S^{-1} N M^{-1} (r S)^\eps.
\label{e467T=Bd}
\end{equation}

Next, we proceed with $\Tn(u;s')$ along the same lines as we did with $\Tn(u)$, keeping in mind that $s' = s_1'$ are fixed. We arrive again at the product of Kloosterman sums, say $V_l(rw)$ with $w = [s,s',s_1,s_1'] = [ s,s',s_1]$. Note that $w \ll S^3$, while we had before $w \ll S^4$. This deficit of $s_1'$ in $w$ will translate into an extra saving of a factor $S^{1/2}$ in estimates for the relevant Kloosterman sums. Specifically, the complete sum $V_l(rw)$ factors into the sum of modulus $r$ and the sum $\cJ_l(w;s')$ of modulus $w = [s,s',s_1]$. The one of modulus $r$ gets the same bounds, but $\cJ_l(w;s')$ gets better bounds;
\[
 |\cJ_l(w;s')|
 \leq \tau(w) \prod_{p \mmid s s' s_1} (l,p)^\frac{1}{2} p^\frac{1}{2}
  \prod_{p^2 \mid ss' s_1} p
 \leq \tau(w) (l,w)^\frac{1}{2} (s s' s_1)^\frac{1}{2}
\]
and $|\cJ_0(w)| \leq \tau(w) w$. Hence we derive
\[
 |V_l(rw)|
 \leq \tau(rw) (l,rw)^\frac{1}{2} r^\frac{1}{2} (2S)^\frac{3}{2},
\]
which is better by a factor $(2S)^\frac{1}{2}$ than~\eqref{e458VlBd}, and
\[
 |V_0(rw)|
 \leq \tau(w) w (d,r),
\]
which is the same as~\eqref{e459V0Bd}. Next, following along the same 
lines, we get
\begin{equation}
\begin{aligned}
 \Tn(u)
 &\ll \sum_s \sum_{s_1}
  \frac{\rho(s) \rho(s')}{[s,s']} \frac{\rho(s_1)\rho(s_1')}{[s_1,s_1']}
\\
 & \qquad \qquad
  \dsum_{\substack{hh_1 \neq 0 \\ d \neq 0}} \left|
   \hat{f} \left( \frac{hM}{r[s,s']} \right) \hat{f} \left( \frac{h_1 M}{r [s_1,s_1']} \right)
   \right|
   \left( N \frac{(d,r)}{r} + S^\frac{3}{2} r^\frac{1}{2} \right) \tau_3(rw)
\\
 &\ll r^2 S^2 M^{-2} \big(r^{-1} N + r^\frac{1}{2} S^\frac{3}{2}\big) (rS)^\eps
  \ll r^\frac{5}{2} S^\frac{7}{2} M^{-2} (rS)^\eps
\end{aligned}
\end{equation}
because of~\eqref{e5.3}. Adding this bound to~\eqref{e467T=Bd} we obtain
\begin{equation}
 T(u;s')
 \ll \big( rS^{-1} N M^{-1} + r^\frac{5}{2} S^\frac{7}{2} M^{-2} \big) (rS)^\eps.
\end{equation}
Hence, by~\eqref{e465D11BdCauchy} we obtain
\begin{equation}
 \cD_{11}
 \ll \big( r^{-\frac{3}{2}} S^\frac{1}{2} N^2 M^\frac{1}{2} + r^{-\frac{3}{4}} S^\frac{11}{4} N^\frac{3}{2} \big) (rS)^\eps.
\label{e470D11Bd2}
\end{equation}

Note that, by comparison, the first term in our first bound~\eqref{e464D11Bd1} is better by a factor $S^\frac{1}{2}$; however, the second term in our second bound~\eqref{e470D11Bd2} is better by a factor $S^\frac{1}{4}$.


\subsection{Evaluation of $\cD_2$}
\label{D2}

Recall that $\cD_2$ stands for the cross terms in the dispersion \eqref{eDispSplitD123},
\begin{equation}
 \cD_2
 = \frac{1}{\varphi(r)} \sum_{(n',r)=1} \beta_{n'} \sum_s \sum_{s'} \sum_{(m,r)=1}
  f \left( \frac{m}{M} \right) \sum_{mn \equiv a \, (r)} \beta_n \gamma_s(mn) \gamma_{s'}(mn').
\label{e471D2repeat}
\end{equation}
We execute the summation over $m$ by splitting into residue classes modulo $r [s,s']$ as in $\cD_1$. Apply Poisson's formula for every class, and pull out the contribution of the frequency $h = 0$, getting
\begin{equation}
 \cD_2
 = \cD_{20} + \cD_{21}
\end{equation}
with
\[
 \cD_{20}
 = \hat{f}(0) \frac{M}{r \varphi(r)} \sum_n \sum_{n'} \beta_n \beta_{n'}
  \sum_s \sum_{s'} \frac{1}{[s,s']} \sum_{\omega \, ([s,s'])} \gamma_s(\omega n) \gamma_{s'}(\omega n').
\]
Do not forget the hidden co-primality conditions $(n,rs) = (n',rs') = 1$. The inner sum depends only on $n'/n$ modulo $c = (s,s')$, see~\eqref{e440gammasbyn'e}. Hence
\[
 \cD_{20}
 = \hat{f}(0) \frac{M}{r\varphi(r)} \sum_s \sum_{s'} \sum_{e \, (c)} \sigma_{s,s'} (e)
  \dsum_{n' \equiv e n \, (c)} \beta_n \beta_{n'}.
\]
Following the arguments applied to $\cD_{10}$ we arrive at
\begin{equation}
 \cD_{20}
 = \cD_{20}^* + \cD_{20}^+,
\end{equation}
where $\cD_{20}^*$ is the main term which agrees with $\cD_{10}^*$ given 
by~\eqref{e445D10*def} and
\begin{multline}
 \cD_{20}^+
 = \hat{f} (0) \frac{M}{r \varphi(r)} \sum_s \sum_{s'} \sumst_{v \, (c)} \sumst_{v' \, (c)} \sigma_{s,s'} (v'/v)
\\
  \Bigg( \sum_{\substack{(n,rs)=1 \\ n \equiv v \, (c)}} \beta_n
   - \frac{1}{\varphi(c)} \sum_{(n,rs)=1} \beta_n
   \Bigg)
  \Bigg( \sum_{\substack{(n',rs')=1 \\ n' \equiv v' \, (c)}} \beta_{n'}
   - \frac{1}{\varphi(c)} \sum_{(n',rs')=1} \beta_{n'}
   \Bigg).
\end{multline}
Following the arguments applied to $\cD_{10}^+$ we conclude the clear estimate
\[
 |\cD_{20}^+|
 \leq |\hat{f}(0)| \frac{M}{r\varphi(r)} \mathop{\sum \sum \sum}_{c,k,l < 2S} \frac{\rho^2(c) \rho(kl)}{ckl}
  \sumst_{\nu \, (c)} \Bigg(
   \sum_{\substack{(n,crk)=1 \\ n \equiv \nu \, (c)}} \beta_n
   - \frac{1}{\varphi(c)} \sum_{(n,crk)=1} \beta_n
   \Bigg)^2.
\]
Hence, 
\begin{equation}
 \cD_{20}^+
 \ll \frac{M}{r\varphi(r)} \bigg(\sum_{l<2S}\frac{\tau(l)^j}{l}\bigg) 
\sum_{k<2S}\frac{\tau(k)^j}{k}\cE_{kr}(1) ,
\end{equation}
which yields the second part on the right side of~\eqref{eDispBd}.


\subsection{Estimation of $\cD_{21}$}
\label{s48estD21}

Recall that $\cD_{21}$ is derived from~\eqref{e471D2repeat} after applying Poisson's formula and dropping the contribution of the frequency $h = 0$. Therefore
\begin{multline}
 \cD_{21}
 = \frac{M}{r\varphi(r)} \dsum_{(nn',r)=1} \beta_n \beta_{n'}
  \sum_s \sum_{s'} \sum_{\omega \, ([s,s'])} \gamma_s(\omega n) 
\gamma_{s'}(\omega n') \big/ [s,s']
\\
  \sum_{h \neq 0} \hat{f} \left( \frac{hM}{r [s,s']} \right)
  \e \left( -ah \frac{\conj{n [s,s']}}{r} 
- \frac{\omega h \conj{r}}{[s,s']} \right).
\end{multline}
Note that this is the same formula as that for $\cD_{11}$, except that here 
the condition $n' \equiv n \pmod{r}$ is gone, but its absence is paid back 
by the factor $1 / \varphi(r)$. Now, the corresponding 
set of shifts $\cU$ is the full 
segment $|u| < 2N$ without the condition $u \equiv 0 \, \pmod{r}$, which 
does not matter for the arguments. Having said that, we get the 
same estimates for $\cD_{21}$ 
as those for $\cD_{11}$. That is, \eqref{e464D11Bd1} and~\eqref{e470D11Bd2} 
are satisfied by $\cD_{21}$.


\subsection{Evaluation of $\cD_3$}
\label{D3}

Recall that $\cD_3$ is derived by squaring the main term in the dispersion~\eqref{eDispSplitD123}:
\begin{equation}
 \cD_3
 = \frac{1}{\varphi(r)^2} \sum_s \sum_{s'} \dsum_{\substack{(n,rs)=1 \\ (n',rs')=1}} \beta_n \beta_{n'}
  \sum_{(m,r)=1} f \left( \frac{m}{M} \right) \gamma_s(mn) \gamma_{s'}(mn').
\end{equation}
Clearly $\cD_3$ is just the average of $\cD_2 = \cD_2(a)$ given by~\eqref{e471D2repeat} over the classes $a \pmod{r}$, $(a,r) = 1$. Since the results for $\cD_2 = \cD_2(a)$ do not depend on $a$ we get the same main term
\begin{equation}
 \cD_{30}^*
 = \cD_{20}^*
 = \cD_{10}^*
\end{equation}
given by~\eqref{e445D10*def} and the same existing bounds 
for the rest of $\cD_3$. Actually, one can get much stronger estimates, 
but we do not need them.

Adding and subtracting all the parts of the dispersion, we complete the 
proof of Proposition~\ref{p43DispBdE} and its 
consequences Proposition~\ref{pEstMaxE} and Corollary~\ref{corEstBilForm}.


\section{\bf{Cropped Divisors in Residue Classes}}
\label{c6CroppDiv}

\subsection{Statement of results}

It has been established in~\cite{FI1} that the divisor function $\tau_3(k)$ 
is distributed uniformly over the residue classes $a \pmod{q}$, $(a,q)=1$ 
to modulus $q$ which is fixed, yet relatively large. One of the results 
of~\cite{FI1} yields
\[
 \sum_{\substack{k \leq x \\ k \equiv a \,(\mo q)}} \tau_3(k)
 - \frac{1}{\varphi(q)} \sum_{\substack{k \leq x \\ (k,q)=1}} \tau_3(k)
 \ll q^{-1} x^{1 - \vartheta}
\]
with $\vartheta > 0$ a small constant, provided 
$q \leq x^{\frac{1}{2} + \frac{1}{231}}$.
In practice one needs similar results for cropped divisors
\[
 \tau_F(k)
 = \sum_{l m n = k} F \left( \frac{l}{L}, \, \frac{m}{M}, \, \frac{n}{N} \right)
\]
where $F(x,y,z)$ is a nice continuous function, compactly supported in 
$\R^+ \times \R^+ \times \R^+$ and $L,M,N \geq 1$ are at our disposal. 
Here the three divisors run over segments $l \asymp L$, $m \asymp M$, 
$n \asymp N$. This extra localization is welcome for applications. 
Actually, we also treated in~\cite{FI1} cropped divisor functions of 
various shapes in a subsidiary capacity. 

In this section we follow closely the arguments of~\cite{FI1} and 
enhance them by a device used by Y. Zhang~\cite{Z} in order to cover 
a bit larger, yet vital range of $L,M,N$. For sheer curiosity, and 
perhaps for clarity, we let our arguments handle technical barriers 
somewhat differently than in~\cite{FI1} and~\cite{Z}. 
First, we fix a cropping function $f$ which is of ${\mathcal C}^4$ class and 
compactly supported in~$\R^+$. Let $q > 1$ be squarefree and $(a,q) = 1$. 
We shall evaluate the sum
\begin{equation}
 S(L,M,N)
 = \tsum_{l m n \equiv a \, (q)} f \left( \frac{l}{L} \right) 
f \left( \frac{m}{M} \right) f \left( \frac{n}{N} \right)
\end{equation}
which is expected to be well-approximated by
\begin{equation}
 S_0(L,M,N)
 = \frac{1}{\varphi(q)} \tsum_{(lmn,q) = 1} f \left( \frac{l}{L} \right) 
f \left( \frac{m}{M} \right) f \left( \frac{n}{N} \right).
\end{equation}
Our goal is to estimate the error term
\begin{equation}
 E(L,M,N)
 = S(L,M,N) - S_0(L,M,N).
\end{equation}
As in~\cite{Z}, we assume that the modulus $q$ has a relatively small factor which will help us cover slightly wider ranges of $L,M,N$.

\begin{prop} \label{p51EBd}
Let $L \geq M \geq N \geq 1$ with $LMN = x$ and $q > 1$ be squarefree, $(a,q) = 1$. Suppose $s \mid q$ and
\begin{equation}
 sN
 < \min (q, x/q).
\label{e54HypEBd}
\end{equation}
Then
\begin{equation}
 E(L,M,N)
 \ll x^\eps \left( \frac{x}{N} \right)^\frac{1}{2} \left( \frac{q}{s} \right)^\frac{1}{4}
  + x^\eps (xsN)^\frac{1}{3}
\label{e55EBd}
\end{equation}
for any $\eps > 0$, the implied constant depending only on $\eps$.
\end{prop}

If $s = 1$, then~\eqref{e55EBd} is covered by Theorem 4 of~\cite{FI1}. We wish the bound~\eqref{e55EBd} to be meaningful (non-trivial); specifically, we want
\begin{equation}
 E(L,M,N)
 \ll q^{-1} x^{1-\eps}
\label{e56EBdGood}
\end{equation}
for some $\eps > 0$, regardless of how small it is. By~\eqref{e55EBd} we get~\eqref{e56EBdGood} if
\begin{equation}
 s^\frac{1}{2} q^\frac{5}{2} x^{\eps - 1}
 < sN
 < x^{-\eps} \min(q, x^2 q^{-3}).
\label{e57EBdGoodConds}
\end{equation}

\begin{cor}\label{c53EBdGood}
Let $q$ be squarefree and $(a,q) = 1$. Suppose $q$ has divisors of any size up to a factor $z \geq 2$. Then~\eqref{e56EBdGood} holds for $L \geq M \geq N \geq 1$ with $LMN = x$ if
\begin{equation}
 q
  \leq \left( \frac{N}{z} \right)^\frac{1}{8} x^{\frac{1}{2} - \eps}.
\end{equation}
\end{cor}

\begin{proof}
We can assume that $z \leq N \leq q x^{-\eps}$, otherwise the result is classical (easy to prove). Since $N \leq x^\frac{1}{3}$ and $q \leq x^\frac{5}{8}$, we have $N < x^{-\eps} \min(q,x^2 q^{-3})$. Therefore, there exists $s \mid q$ with
\[
 \frac{\min}{x^\eps z N}
 \leq s
 < \frac{\min}{x^\eps N}.
\]
For this choice of $s$ the conditions~\eqref{e57EBdGoodConds} are satisfied if
\[
 q^\frac{5}{2} x^{\eps - 1}
 < N \left( \frac{\min}{x^\eps z N} \right)^\frac{1}{2},
\]
that is, if $x^{3 \eps} z q^4 < N x^2$ and $x^\eps z q^8 < N x^4$. However, the first condition follows from the second one. This completes the proof.
\end{proof}

Since Corollary~\ref{c53EBdGood} holds for every class $a \pmod{q}$ with $(a,q) = 1$, the result can be automatically generalized as follows:

\begin{cor} \label{cor53estSumFlmn}
Let $q$ be squarefree and $(ad,q) = 1$. Suppose $q$ has divisors of any size up to a factor $z \geq 2$. Then we have
\begin{equation}
 \tsum_{dlmn \equiv a \, (q)} f \left( \frac{l}{L} \right) f \left( \frac{m}{M} \right) f \left( \frac{n}{N} \right)
 = S_0(L,M,N)
  + O(q^{-1} d^{-1} x^{1-\eps})
\end{equation}
for any $L \geq M \geq N \geq 1$ with $dLMN = x$, if
\begin{equation}
 q \leq \left(\frac{N}{z}\right)^\frac{1}{8} \left(\frac{x}{d}\right)^{\frac{1}{2}} x^{-\eps}.
\label{esixten}
\end{equation}
\end{cor}


\subsection{Extracting the main term}

We begin the proof of Proposition~\ref{p51EBd} by applying Poisson's formula in the largest variable $l \equiv a \conj{mn} \pmod{q}$, getting
\begin{equation}
 S(L,M,N)
 = \frac{L}{q} \sum_h \hat{f} \left(\frac{hL}{q} \right)
  \dsum_{mn} f \left(\frac{m}{M}\right) f\left(\frac{n}{N}\right) \e\left(ah \frac{\conj{mn}}{q}\right).
\label{e510SPoissL}
\end{equation}
Here and hereafter, we use the convention that if the multiplicative inverse 
of $d$ modulo $q$ appears, then it is assumed automatically, without writing, 
that $d$ is restricted to $(d,q) = 1$. Of course, one must recall this 
restriction when things go out of sight.

The main term emerges from the zero frequency, which yields
\begin{equation}
 \hat{f}(0) \frac{L}{q} \dsum_{(mn,q)=1} f 
\left(\frac{m}{M}\right) f \left(\frac{n}{N}\right).
\label{e511ZeroFreq}
\end{equation}
This is not exactly $S_0(L,M,N)$, but close. Indeed, we have
\begin{align*}
 \sum_{(l,q)=1} f\left(\frac{l}{L}\right)
 &= \sum_{\delta \mid q} \mu(\delta) \sum_l f \left( \frac{\delta l}{L}\right)
  = \sum_{\delta \mid q} \mu(\delta) \left( \hat{f}(0) \frac{L}{\delta} + O(1)\right)
\\
 &= \hat{f}(0) L \frac{\varphi(q)}{q}
   + O\big(\tau(q)\big).
\end{align*}
Hence~\eqref{e511ZeroFreq} becomes $S_0(L,M,N) + O(MN \tau(q) / \varphi(q) )$, and we are left with
\begin{equation}
 S(L,M,N)
 = S_0(L,M,N)
  + S^*(L,M,N)
  + O\big( MN \tau(q) / \varphi(q) \big),
\label{e512Ssplit}
\end{equation}
where $S^*(L,M,N)$ denotes the part of~\eqref{e510SPoissL} with $h \neq 0$. This will be estimated in several steps.


\subsection{Weyl's shift}

First we are going to replace $f(m/M)$ by $f((m+bh)/M)$, where $b$ is a positive integer at our disposal; it will be chosen later from a finite set of integers of suitable type. We write
\begin{equation}
 S^*(L,M,N)
 = S_b (L,M,N) - \partial S_b (L,M,N),
\label{e513S*split}
\end{equation}
where
\begin{equation}
 S_b(L,M,N)
 = \frac{L}{q} \sum_n f \left(\frac{n}{N}\right)
  \sum_m \sum_{h \neq 0} \hat{f} \left( \frac{hL}{q} \right)
  f \left( \frac{m+bh}{M}\right) \e \left( ah \frac{\conj{mn}}{q} \right)
\label{e514SbDef}
\end{equation}
and $\partial S_b(L,M,N)$ is similar, but with the inner sum over $h \neq 0$ replaced by
\begin{equation}
 \sum_h \hat{f} \left( \frac{hL}{q} \right) \left(
  f \left( \frac{m+bh}{M} \right)
  - f \left( \frac{m}{M} \right)
  \right)
  \e \left( ah \frac{\conj{mn}}{q} \right).
\label{e515PartSbInnSum}
\end{equation}
Applying Poisson's formula to~\eqref{e515PartSbInnSum} (a reversed 
application) we find that the sum~\eqref{e515PartSbInnSum} is equal to
\begin{equation}
 \frac{q}{L} \sum_{l \equiv a \conj{mn} \, (q)} g \left( \frac{l}{L} \right),
\label{e516PartSbInnSumRevPoiss}
\end{equation}
where $g(u)$ is such that
\[
 \hat{g} \left( \frac{hL}{q} \right)
 = \hat{f} \left( \frac{hL}{q} \right)
  \left(
  f \left( \frac{m+bh}{M} \right)
  - f \left( \frac{m}{M} \right)
  \right)
\]
for every real $h$, that is;
\begin{equation}
 \hat{g}(v)
 = \hat{f}(v) \left(
  f \left( \frac{m}{M} + v \frac{bq}{LM} \right)
  - f \left( \frac{m}{M} \right)
  \right).
\end{equation}
By partial integration and the mean-value theorem we derive the following estimates:
\[
 |\hat{g}(v)|
 + |\hat{g}(v)''|
 \ll \frac{bq}{LM} (1 + v^2)^{-2},
  \qquad
 g(u)
 \ll \frac{bq}{LM} (1 + u^2)^{-1}.
\]
Hence, \eqref{e516PartSbInnSumRevPoiss} yields
\[
 \partial S_b(L,M,N)
 \ll \frac{bq}{LM} \tsum_{\substack{lmn \equiv a \, (q) \\ m \asymp M, \, n \asymp N}} (1 + l/L)^{-2}
 \ll x^\eps b N,
\]
by using $\tau_3(k) \ll k^\eps$. Inserting this estimate 
into~\eqref{e513S*split} we get
\begin{equation}
 S^*(L,M,N)
 = S_b(L,M,N) + O(x^\eps b N)
\label{e518S*estPartSb}
\end{equation}
for any $\eps > 0$, the implied constant depending on $\eps$. Note that $S^*(L,M,N)$ does not depend on $b$, whereas the right side of~\eqref{e518S*estPartSb} has received the extra parameter $b$ for free.

In~\eqref{e514SbDef} we translate $m$ to $m-bh$ and we write $h/q$ 
in the lowest terms, getting
\[
 S_b(L,M,N)
 = \frac{L}{q} \sum_{q_0 q_1 = q} \sum_{\substack{h \neq 0 \\ (h,q_1) =1}}
  \hat{f} \left( \frac{hL}{q_1} \right)
  \dsum_{(mn,q_0)=1} f \left( \frac{m}{M} \right) f \left( \frac{n}{N} \right)
  \e \left( \frac{ah}{q_1} \conj{(m + bh q_0) n} \right).
\]
Here the co-primality condition $(m,q_0) = 1$ is harmless, 
but $(n,q_0)=1$ should be relaxed before applying Cauchy's inequality. 
To this end we utilize the Möbius formula, giving
\begin{multline*}
 S_b(L,M,N)
 = \frac{L}{q} \sum_{q_0 q_1 = q} \sum_{\eta \mid q_0} \mu(\eta)
  \sum_{\substack{h \neq 0 \\ (h,q_1) = 1}} \hat{f} \left( \frac{hL}{q_1} \right)
  \sum_{(m,q_0) = 1} f \left( \frac{m}{M} \right)
\\
  \sum_n f \left( \frac{n \eta}{N} \right)
  \e \left( \frac{a \conj{q_0 \eta}}{q_1} \conj{ (m \conj{hq_0} + b) n} \right).
\end{multline*}


\subsection{Compacting variables}

Our next step (inspired by Burgess' idea) is to build one larger variable out of $m$ and $h$. Putting $a_0 \equiv a \conj{q_0 \eta} \pmod{q_1}$ and for any $x \pmod{q_1}$
\[
 \nu(x)
 = \sum_{\substack{h \neq 0 \\ (h,q_1)=1}} \hat{f} \left( \frac{hL}{q_1} \right)
  \sum_{\substack{(m,q_0)=1 \\ m \equiv x h q_0 \, ( \mo q_1)}} f \left( \frac{m}{M} \right),
\]
so we can write
\[
 S_b(L,M,N)
 = \frac{L}{q} \sum_{q_0 q_1 = q} \sum_{\eta \mid q_0} \mu(\eta)
 \sum_{x \, (q_1)} \nu(x)
 \sum_n f \left( \frac{n \eta}{N} \right)
 \e \left( \frac{a_0}{q_1} \conj{(x+b) n} \right).
\]
Note that
\begin{equation}
 \sum_{x \, (q_1)} \nu(x)^2
 \ll \frac{M}{L} q_1^{1+\eps}
\label{e519SumNu2Bd}
\end{equation}
because the number of solutions of $h' m \equiv h m' \pmod{q_1}$ in $1 \leq h,h' \leq H$ and $1 \leq m,m' \leq M$, for every $H,M \geq 1$ is bounded by $HM (1 + HM / q_1) q_1^\eps$. Now we choose a finite set $\cB$ of positive integers and sum $S_b(L,M,N)$ over $b \in \cB$. Then we apply Cauchy's inequality and~\eqref{e519SumNu2Bd} getting
\begin{equation}
 \sum_{b \in \cB} S_b(L,M,N)
 \ll q^\eps \sum_{q_0 q_1 = q} \left( \frac{LM}{q_0 q} \right)^\frac{1}{2}
  \sum_{\eta \mid q_0} T( \cB, N / \eta)^\frac{1}{2},
\label{e520SumSb}
\end{equation}
where
\[
 T(\cB,N)
 = \sum_{x \, (\mo q_1)} \left|
  \sum_n f \left( \frac{n}{N} \right)
  \sum_{b \in \cB} \e \left( \frac{a_0}{q_1} \conj{(x+b) n} \right) \right|^2.
\]
Note that we do not display (for brevity) all the floating parameters upon 
which our objects depend. For instance, $T(\cB,N)$ depends on the modulus 
$q_1$ and also on $q_0$, $\eta$ via $a_0 \equiv a \conj{q_0 \eta} \, 
\pmod{q_1}$. 
Squaring out we find
\begin{equation}
 T(\cB,N)
 = T^=(\cB,N)
  + \Tn(\cB,N),
\end{equation}
where
\begin{equation}
 \Tn(\cB,N)
 = \dsum_{b \neq b'} \sum_n \sum_{n'}
  f \left( \frac{n}{N} \right)
  f \left( \frac{n'}{N} \right)
  \sum_{x \, (q_1)} \e \left( \frac{a_0}{q_1} \left( \conj{(x + b - b')n} - \conj{xn'} \right) \right)
\label{e522TnDef}
\end{equation}
and $T^=(\cB,N)$ is the corresponding sum with $b = b'$. In this case the inner sum reduces to the Ramanujan sum (remember $(nn', q_1) = 1$ by our convention)
\[
 \sumst_{x \, (q_1)} \e \big( (n'-n) \frac{x}{q} \big)
 \ll (n'-n, q_1).
\]
Summing over $n$, $n'$, we obtain
\begin{equation}
 T^=(\cB,N)
 \ll \tau(q) N(q_1 + N) |\cB|.
\label{e523T=Bd}
\end{equation}
The other sum $\Tn(\cB,N)$ is much more difficult to estimate; we 
apply the Riemann hypothesis for varieties in the concealed form of bounds 
for 3-dimensional exponential sums.


\subsection{Correlation of Kloosterman sums}

In~\eqref{e522TnDef} we execute the summation over $n$, $n'$ by Poisson's 
formula, obtaining
\begin{equation}
 \Tn(\cB,N)
 = \left( \frac{N}{q_1} \right)^2 \dsum_{b \neq b'} \sum_v \sum_w
  \hat{f} \left( \frac{vN}{q_1} \right) \hat{f} \left( \frac{wN}{q_1} \right)
  K( b-b', a_0 v, a_0 w; q_1),
\label{e524TnPoiss}
\end{equation}
where $K(u,v,w;q_1)$ denotes the complete sum modulo $q_1$ in three variables (a kind of correlation of Kloosterman sums)
\begin{equation}
 K(u,v,w;k)
 = \tsum_{x,y,z \, (\mo k)}
 \e_k \left( \conj{(x+u)y} - \conj{xz} + vz + wy \right).
\label{e525DefK}
\end{equation}

\begin{prop} \label{p54KBd}
Let $k \geq 1$ be squarefree and $u,v,w$ be integers. We have
\begin{equation}
 K(u,v,w;k)
 \ll (u,k)^\frac{1}{2} (vw,k)^{-\frac{1}{2}} k^{\frac{3}{2} + \eps},
\label{e526KuvwkBd}
\end{equation}
where $\eps$ is any positive number, and the implied constant depends only on $\eps$. Moreover, for $u \equiv 0 \pmod{k}$ we have
\begin{equation}
 |K(0,v,w;k)|
 \leq k (v+w,k) / (vw,k).
\label{e527estK0vw}
\end{equation}
\end{prop}

\begin{proof}
By multiplicativity we can assume that $k = p$ is prime. We shall compute $K(u,v,w;p)$ in all cases, except for $p \nmid uvw$. If $p \mid u$ we get
\begin{align*}
 K(0,v,w;p)
 &= \sum_x \sum_y \sum_z \e_p \left( \conj{xy} - \conj{xz} + vz + wy \right)
\\
 &= R(v+w,p)p - R(v,p) R(w,p)
  =
  \begin{cases}
   R(v+w,p), & \text{if } p \mid vw, \\
   R(v+w,p)p -1, & \text{if } p \nmid vw,
  \end{cases}
\end{align*}
where $R(a,p)$ is the Ramanujan sum; $|R(a,p)| < (a,p)$. Hence
\[
 |K(0,v,w;p)|
 < p(v+w,p) / (vw,p)
\]
which proves~\eqref{e527estK0vw}. Now let $p \nmid u$, $p \mid v$. We get
\begin{align*}
 K(u,0,w;p)
 &= \sum_x \sum_y \sum_z \e_p \left( \conj{(x+u)y} - \conj{xz} + wy \right)
\\
 &= - \sum_y \sum_{x \not\equiv 0} \e_p \left( \conj{(x+u)y} + wy \right)
  = K(\conj{u},w;p) + R(w,p),
\end{align*}
where $K(\conj{u},w;p)$ denotes the classical Kloosterman sum. Similarly 
we evaluate \break $K(u,v,0;p) = K(\conj{u},v;p) + R(v,p)$. In both cases, 
using Weil's bound we obtain
\[
 |K(u,v,w;p)|
 \leq 3(v+w,p)^\frac{1}{2} p^\frac{1}{2},
  \qquad \text{if } p \nmid u, \, p \mid vw.
\]
This bound easily satisfies~\eqref{e526KuvwkBd}. Finally, if $p \nmid uvw$ 
we cannot compute $K(u,v,w;p)$, but we have the bound (see the Appendix 
in~\cite{FI1} by Birch and Bombieri)
\[
 K(u,v,w;p)
 \ll p^\frac{3}{2}.
\]
Clearly, this bound satisfies~\eqref{e526KuvwkBd}. Combining these results 
one completes the proof of Proposition~\ref{p54KBd}. 
\end{proof}

\begin{rem}
Sums of type~\eqref{e525DefK} have been recently considered among much more 
general sums in a conceptual framework by~\cite{FKM1}.
\end{rem}


\subsection{Completing the proof of Proposition~\ref{p51EBd}}

Now, we proceed to the estimation of $\Tn(\cB,N)$, starting from the 
formula~\eqref{e524TnPoiss}, by applying Proposition~\ref{p54KBd}. This 
constitutes the core of the proof of Proposition~\ref{p51EBd}. At this point, 
we are going to exploit the divisor $s$ of $q$ and, as Zhang~\cite{Z}, we 
take the set $\cB$ of positive integers which are divisible by $s$;
\[
 \cB
 = \{ b = cs \, ; \, 1 \leq c \leq C \},
\]
where $C = |\cB|$ will be chosen later to optimize the results. The case 
$s = 1$ (our original setup in~[FI3]) 
does lead to pretty strong results. However, by allowing $s$ to be a 
nontrivial divisor of $q$, Zhang succeeded to produce improvements, 
crucial for this application, which are due to the estimates for 
the Ramanujan sums being better than those for the correlation 
Kloosterman sums (compare~\eqref{e527estK0vw} 
and~\eqref{e526KuvwkBd}).

Put $s_1 = s / (s,q_0)$ and $r = q / [s,q_0] = q (s,q_0) / s q_0$. We have
\[
 q_1 = r s_1,
\qquad
 s_1 \mid s,
\qquad
 (r,s)=1.
\]
The complete sum of modulus $q_1$ in~\eqref{e524TnPoiss} factors into the 
complete sum of modulus $r$ and the complete sum of modulus $s_1$, but with 
the exponentials (additive characters) twisted by the factors 
$\conj{s_1} \pmod{r}$ and $\conj{r_1} \pmod{s_1}$, respectively. 
By~\eqref{e526KuvwkBd} we get
\[
 K(*,*,*;r)
 \ll (c-c',r)^\frac{1}{2} (vw,r)^{-\frac{1}{2}} r^{\frac{3}{2} + \eps},
\]
and by~\eqref{e527estK0vw} we get
\[
 K(*,*,*;s_1)
 \ll s_1 (v+w,s_1) / (vw,s_1).
\]
Inserting these estimates into~\eqref{e524TnPoiss} and summing over 
$1 \leq c \neq c' \leq C$ we obtain
\begin{equation}
 \Tn(\cB,N)
 \ll q^\eps \frac{C^2 N^2}{q_1}
  \sum_v \sum_w \left| \hat{f} \left( \frac{vN}{q_1} \right) 
\hat{f} \left(\frac{wN}{q_1}\right) \right|
  \frac{(v+w,s_1)}{(vw,s_1)} \left( \frac{r}{(vw,r)} \right)^\frac{1}{2}.
\label{e528TnBd1}
\end{equation}
The terms $v = w = 0$ contribute $|\hat{f}(0)|^2$. The terms $v = 0$, 
$w \neq 0$, and the terms $v \neq 0$, $w = 0$ contribute
\[
 2 |\hat{f}(0)| \sum_{w \neq 0} \left| \hat{f} 
\left( \frac{wN}{q_1} \right) \right|
  \frac{(w,s_1)}{s_1}
 \ll \frac{q_1 \tau(s_1)}{N s_1}.
\]
The terms $v = -w \neq 0$ contribute at most
\[
 2 r^\frac{1}{2} s_1 \sum_{w \neq 0} 
\left| f \left( \frac{wN}{q_1} \right) \right|^2
 \ll r^\frac{1}{2} s_1 q_1 / N.
\]
The terms $vw (v+w) \neq 0$ contribute at most
\[
 r^\frac{1}{2} \dsum_{vw(v+w) \neq 0} \left|
  \hat{f} \left(\frac{vN}{q_1} \right) \hat{f} \left( \frac{wN}{q_1} \right)
  \right|
  (v+w,s_1).
\]
Here, we apply the following estimates for the Fourier transforms
\[
 |\hat{f}(v) \hat{f}(w)|
 \ll (1 + |v|)^{-2} (1 + |w|)^{-4}
 \leq (1 + |v+w|)^{-2} (1 + |w|)^{-2}.
\]
Hence, the contribution of terms $vw (v+w) \neq 0$ is bounded by
\[
 r^\frac{1}{2} \sum_{w > 0} \left( 1 + \frac{wN}{q_1} \right)^{-2}
  \sum_{t > 0} \left( 1 + \frac{tN}{q_1} \right)^{-2} (t,s_1)
 \ll \tau(s_1) r^\frac{1}{2} q_1^2 N^{-2}.
\]
Inserting the above estimates into~\eqref{e528TnBd1} we get
\[
 \Tn(\cB,N)
 \ll q^\eps \frac{C^2 N^2}{q_1} \left\{
  1
  + \frac{q_1}{N s_1}
  + \frac{r^\frac{1}{2} s_1 q_1}{N}
  + \frac{r^\frac{1}{2} q_1^2}{N^2}
  \right\}
 \ll C^2 s^{-\frac{1}{2}} q^{\frac{3}{2} + \eps}, 
\]
because $sN < q$ by~\eqref{e54HypEBd}. 
Adding~\eqref{e523T=Bd}, we obtain
\[
 T(\cB,N)
 \ll q^\eps ( CNq + C^2 s^{-\frac{1}{2}} q^\frac{3}{2} ).
\]
Next, inserting this into~\eqref{e520SumSb} we have
\[
 \frac{1}{|\cB|} \sum_{b \in \cB} S_b(L,M,N)
 \ll q^\eps \left( \frac{LMN}{C} \right)^\frac{1}{2}
  + q^\eps \left( \frac{q}{s} \right)^\frac{1}{4} (LM)^\frac{1}{2}.
\]
Adding this to~\eqref{e518S*estPartSb} we find
\[
 S^*(L,M,N)
 \ll x^\eps \left(
  \left( \frac{x}{c} \right)^\frac{1}{2}
  + \left( \frac{q}{s} \right)^\frac{1}{4} \left( \frac{x}{N} \right)^\frac{1}{2}
  + C s N
  \right).
\]
Here $C$ is any positive integer, but, of course, the result holds for any real number $C \geq 1$. We take $C = (x/ S^2 N^2)^{1/3} \geq 1$ by~\eqref{e54HypEBd}, getting
\[
 S^*(L,M,N)
 \ll x^\eps \left( \frac{q}{s} \right)^\frac{1}{4} \left(\frac{x}{N}\right)^\frac{1}{2}
  + x^\eps (xsN)^\frac{1}{3}.
\]
Hence, by~\eqref{e512Ssplit} we arrive at~\eqref{e55EBd}, which completes the proof of Proposition~\ref{p51EBd}.




\begin{bibdiv}
\begin{biblist}

\bib{BD}{article}{
   author={Bombieri, E.},
   author={Davenport, H.},
   title={Small differences between prime numbers},
   journal={Proc. Roy. Soc. Ser. A},
   volume={293},
   date={1966},
   pages={1--18},
}


\bib{BFI1}{article}{
    author={Bombieri, E.},
    author={Friedlander, J. B.}, 
    author={Iwaniec, H.},
	TITLE = {Primes in arithmetic progressions to large moduli},
	JOURNAL = {Acta Math.},
    VOLUME = {156},
	YEAR = {1986},
    NUMBER = {3-4},
    PAGES = {203--251},
}

\bib{BFI2}{article}{
    author={Bombieri, E.}, 
    author={Friedlander, J. B.},  
    author={Iwaniec, H.},
    TITLE = {Primes in arithmetic progressions to large moduli. {II}},
	JOURNAL = {Math. Ann.},
    VOLUME = {277},
    YEAR = {1987},
    NUMBER = {3},
    PAGES = {361--393},
}

\bib{BFI3}{article}{
    author={Bombieri, E.}, 
    author={Friedlander, J. B.},
    author={Iwaniec, H.},
    TITLE = {Primes in arithmetic progressions to large moduli. {III}},
	JOURNAL = {J. Amer. Math. Soc.},
	VOLUME = {2},
    YEAR = {1989},
    NUMBER = {2},
    PAGES = {215--224},
}
		
\bib{E}{article}{
    author={Erd{\H o}s, P.},
    TITLE = {The difference of consecutive primes},
    JOURNAL = {Duke Math. J.},
    VOLUME = {6},
    YEAR = {1940},
    PAGES = {438--441},
}

\bib{F}{article}{
    author={Fouvry, E.},
    TITLE = {Autour du th\'eor\`eme de {B}ombieri-{V}inogradov},
	JOURNAL = {Acta Math.},
    VOLUME = {152},
    YEAR = {1984},
    NUMBER = {3-4},
    PAGES = {219--244},
}

\bib{FoIw}{article}{
    author={Fouvry, E.}, 
	author={Iwaniec, H.},
    TITLE = {Primes in arithmetic progressions},
    JOURNAL = {Acta Arith.},
    VOLUME = {42},
    YEAR = {1983},
    NUMBER = {2},
    PAGES = {197--218},
}

\bib{FKM1}{article}{
	author={Fouvry, E.},
	author={Kowalski, E.}, 
	author={Michel, P.},
	TITLE = {Algebraic twists of modular forms and Hecke orbits}, 
        YEAR = {2012},         
	note = {Preprint},         
}

 \bib{FKM2}{article}{
	author={Fouvry, E.},
	author={Kowalski, E.}, 
	author={Michel, P.},
	TITLE = {On the exponent of distribution of the ternary divisor function},
        YEAR = {2013}, 
        note = {Preprint}, 
}

\bib{FI1}{article}{
    author={Friedlander, J. B.}, 
	author={Iwaniec, H.},
    TITLE = {Incomplete {K}loosterman sums and a divisor problem},
	JOURNAL = {Ann. of Math. (2)},
	VOLUME = {121},
    YEAR = {1985},
    NUMBER = {2},
    PAGES = {319--350},
}

\bib{FI2}{book}{
    author={Friedlander, J. B.}, 
	author={Iwaniec, H.},
    TITLE = {Opera de cribro},
    SERIES = {American Mathematical Society Colloquium Publications},
	VOLUME = {57},
	PUBLISHER = {American Mathematical Society},
    ADDRESS = {Providence, RI},
    YEAR = {2010},
}

\bib{GY}{article}{
    author={Goldston, D. A.},
	author={Y{\i}ld{\i}r{\i}m, C. Y.},
    TITLE = {Higher correlations of divisor sums related to primes. {III}. Small gaps between primes},
	JOURNAL = {Proc. Lond. Math. Soc. (3)},
	VOLUME = {95},
    YEAR = {2007},
    NUMBER = {3},
    PAGES = {653--686},
}

\bib{GPY}{article}{
    author={Goldston, D. A.}, 
	author={Pintz, J.},
	author={Y{\i}ld{\i}r{\i}m, C. Y.},
    TITLE = {Primes in tuples. {I}},
    JOURNAL = {Ann. of Math. (2)},
     VOLUME = {170},
    YEAR = {2009},
    NUMBER = {2},
    PAGES = {819--862},
}

\bib{H}{article}{
    author={Huxley, M. N.},
    TITLE = {On the differences of primes in arithmetical progressions},
    JOURNAL = {Acta Arith.},
      VOLUME = {15},
    YEAR = {1968/1969},
    PAGES = {367--392},
}
		
\bib{M}{article}{
    author={Maier, H.},
    TITLE = {Small differences between prime numbers},
	JOURNAL = {Michigan Math. J.},
      VOLUME = {35},
    YEAR = {1988},
    NUMBER = {3},
    PAGES = {323--344},
}

\bib{MP}{article}{
    author={Motohashi, Y.}, 
	author={Pintz, J.},
    TITLE = {A smoothed {GPY} sieve},
    JOURNAL = {Bull. Lond. Math. Soc.},
    VOLUME = {40},
    YEAR = {2008},
    NUMBER = {2},
    PAGES = {298--310},
}

\bib{Z}{article}{
	author={Zhang, Y.},
	title = {Bounded gaps between primes},
	note = {To appear in Ann. Math.},
        YEAR = {2013} 
}


\end{biblist}
\end{bibdiv}
\end{document}